\documentclass[11pt]{amsart}
\usepackage{amssymb,amsmath,amsfonts,amscd,euscript}
\usepackage[matrix,arrow,curve]{xy}
\usepackage[backref=page]{hyperref}

\newcommand{\nc}{\newcommand}

\newtheorem{trm}{Theorem}
\newtheorem{con}{Conjecture}
\numberwithin{equation}{section}
\newtheorem{thm}{Theorem}[section]
\newtheorem{prop}[thm]{Proposition}
\newtheorem{lem}[thm]{Lemma}
\newtheorem{cor}[thm]{Corollary}
\theoremstyle{remark}
\newtheorem{rem}[thm]{Remark}

\newtheorem{dfn}[thm]{Definition}

\nc{\fB}{\mathfrak{B}}
\nc{\gl}{\mathfrak{gl}}
\nc{\GL}{\mathfrak{GL}}
\nc{\g}{\mathfrak{g}}
\nc{\gh}{\widehat\g}
\nc{\h}{\mathfrak{h}}
\nc{\wfh}{\widehat{\mathfrak{h}}}
\nc{\la}{\lambda}
\nc{\al}{\alpha }
\nc{\be}{\beta }
\nc{\ve}{\varepsilon }
\nc{\om}{\omega }
\nc{\lr}{\text{-}}
\nc{\ta}{\theta}
\nc{\ch}{{\mathop {\rm ch}}}
\nc{\Tr}{{\mathop {\rm Tr}\,}}
\nc{\Id}{{\mathop {\rm Id}}}
\nc{\ad}{{\mathop {\rm ad}}}
\nc{\bra}{\langle}
\nc{\ket}{\rangle}
\nc{\bi}{{\bf i}}
\nc{\pa}{\partial}
\nc{\ld}{\ldots}
\nc{\cd}{\cdots}
\nc{\hk}{\hookrightarrow}
\nc{\T}{\otimes}
\nc{\gr}{\mathrm{gr}}
\nc{\ov}{\overline}

\nc{\cO}{\mathcal O}
\nc{\msl}{\mathfrak{sl}}
\nc{\mgl}{\mathfrak{gl}}
\nc{\U}{\mathrm U}
\nc{\V}{\EuScript V}
\nc{\cL}{\mathcal{L}}
\nc{\Res}{\mathrm{Res\ }}

\newcommand{\bC}{{\Bbbk}}

\newcommand{\bZ}{{\mathbb Z}}

\newcommand{\fh}{{\mathfrak h}}

\newcommand{\fg}{{\mathfrak g}}

\newcommand{\fb}{{\mathfrak b}}

\newcommand{\fn}{{\mathfrak n}}

\nc{\I}{\mathfrak I}
\nc{\bfI}{\mathbf I}
\nc{\Q}{\mathfrak Q}
\nc{\W}{\mathbb W}
\nc{\bU}{\mathbb U}

\nc{\Gm}{\mathbb{G}_{m}}
\nc{\bA}{\mathbb A}

\newcommand{\ZZ}{\mathbb{Z}}
\newcommand{\QQ}{\mathbb{Q}}
\newcommand{\K}{\mathbb{K}}

\newcommand{\dg}{\mathrm{dg}}

\newcommand{\cA}{\mathcal{A}}

\newcommand{\tE}{\tilde{E}}

\newcommand{\tia}{\tilde{a}}

\newcommand{\lab}{\overline{\lambda}}
\newcommand{\mub}{\overline{\mu}}
\newcommand{\Fb}{\overline{F}}
\newcommand{\one}{\mathbf{1}}

\newcommand{\veb}{\overline{\varepsilon}}

\newcommand{\arm}{\mathrm{arm}}
\newcommand{\leg}{\mathrm{leg}}

\newcommand{\res}{\mathrm{res}}

\begin{document}

\title[]
{Nonsymmetric $q$-Cauchy identity and representations of the Iwahori algebra}

\author{Evgeny Feigin}
\address{Evgeny Feigin:\newline
Department of Mathematics, HSE University, Usacheva str. 6, 119048, Moscow, Russia\newline
Faculty of Mathematics and Computer Science, Weizmann Institute of
Science, POB 26, Rehovot 76100, Israel
}
\email{evgfeig@gmail.com}

\author{Ievgen Makedonskyi}
\address{Ievgen Makedonskyi:\newline
Yanqi Lake Beijing Institute of Mathematical Sciences And Applications, 
11th Building, Yanqi Island, Huairou District, Beijing 101408}
\email{makedonskii\_e@mail.ru}

\author{Daniel Orr}
\address{Daniel Orr:\newline  
Department of Mathematics (MC 0123), 460 McBryde Hall, Virginia Tech, 225 Stanger St., Blacksburg,
VA 24061 USA / Max Planck Institut fur Mathematik, Vivatsgasse 7, 53111 Bonn, Germany}
\email{dorr@vt.edu}

\begin{abstract}
The $t=0$ specialization of the Mimachi-Noumi Cauchy-type identity rewrites certain
infinite product in terms of specialized nonsymmetric Macdonald polynomials of type $GL_n$. 
We interpret the infinite product as a character of the space of functions on a certain matrix space. We show that the space of functions admits a filtration such that the graded pieces are isomorphic to the tensor products
of certain generalized global Weyl modules of the Iwahori algebra. We identify 
the characters of the graded pieces with the terms of the specialized Mimachi-Noumi formula. We conjecture the existence of  an analogous filtration on the space of functions on the Iwahori group  for all simple Lie algebras and prove the conjecture for $SL_n$. Our construction can be seen as a current algebra extension of the van der Kallen filtration on functions on a Borel subgroup. 
\end{abstract}

\maketitle

\section{Introduction}
The celebrated Cauchy identity 
\[
\prod_{1\le i,j\le n} (1-x_iy_j)^{-1} = \sum_{\la=(\la_1\ge \dots\ge \la_n\ge 0)} s_\la(x_1,\dots,x_n) s_\la(y_1,\dots,y_n)
\]
allows to rewrite the infinite product in the left hand side as a sum of products of
Schur functions (see \cite{M}). The identity has a vast number of applications and interpretations (see 
e.g. \cite{BC,BP,BW,Ok,OR,St}). In particular, one can understand the left hand side as the character of polynomial functions on the space of square matrices and decompose this space of functions into the
direct sum of bimodules over the Lie algebra $\mgl_n$. The products of Schur functions
show up in this approach as characters of tensor products of irreducible $\mgl_n$ modules. 
From the point of view of representation theory, the Cauchy identity is a close cousin  of the
Peter-Weyl theorem (see \cite{PW,TY}), since the latter decomposes the space of function on a simple group as
the direct sum of tensor products of irreducible representations.

The classical Cauchy identity admits a generalization in the setting of current algebras. 
The corresponding numerical identity reads as (see \cite{M})
\[
\prod_{k\ge 0}\prod_{1\le i,j\le n} (1-x_iy_jq^k)^{-1} = \sum_{\la=(\la_1\ge \dots\ge \la_n\ge 0)} p_\la(x;q) p_\la(y;q)\prod_{a=1}^n (q)_{\la_a-\la_{a+1}}^{-1},
\]
where $x=(x_1,\dots,x_n)$, $p_\la(x;q)$ denotes the $q$-Whittaker functions 
and $(q)_m$ stands for the $q$-Pochhammer symbol (we note that the $q$-Whittaker functions are equal to the  $t=0$ specializations of the nonsymmetric Macdonald polynomials $E_{w_0\la}(x;q,0)$,  see  \cite{BF,Et,CO1,CO2,GLO}).
In \cite{FKM} the authors suggested
to interpret the left hand side as the character of the space of polynomial functions on the space of square
matrices with coefficients in polynomial ring in one variable. The authors constructed  
a filtration on this space of functions such that the homogeneous components of the associated
graded space are isomorphic to the tensor products (over the highest weight algebras) of global 
Weyl modules. The paper \cite{FKM} also contains a current algebras analogue of the Peter-Weyl
theorem, expressing the characters of the spaces of functions on current groups via products
of $t=0$ specialized Macdonald polynomials. 

The classical Cauchy identity has a non-symmetric version (see \cite{CK,L,MN})
\[
\prod_{1\le i\le j\le n} (1-x_iy_j)^{-1} = \sum_{\la\in(\bZ_{\ge 0})^n} 
E_\la(x;0,0) E_\la(y;\infty,\infty).
\]
Here $E_\la$ stands for the nonsymmetric Macdonald polynomials of type $GL_n$, $E_\la(x;0,0)$ are the
key polynomials (a.k.a. the characters of the Demazure modules) and 
$E_\la(y;\infty,\infty)$ are the Demazure atoms (a.k.a. the characters of the van der Kallen modules -- the  
quotients of the Demazure modules by the sum of all smaller Demazure submodules). A $\mgl_n$-version of a theorem of van der Kallen \cite{vdK} provides a representation-theoretic counterpart of the nonsymmetric
Cauchy identity. Namely, the left hand side is replaced by the character of the space
of polynomial functions on the upper triangular matrices. This space of functions admits
a filtration (indexed by  $\lambda\in\ZZ_{\ge 0}^n$) such that the homogeneous pieces of the associated graded space are isomorphic
to the tensor products of the Demazure modules and the van der Kallen modules. The van der Kallen theorem holds true for all simple Lie groups provided the space of upper triangular
matrices is replaced by the Borel subgroup of the corresponding simple Lie group.

The nonsymmetric Cauchy identity has a $(q,t)$-version due to Mimachi-Noumi (see \cite {MN}). The $t=0$ specialization of the Mimachi-Noumi formula reads as 
\begin{multline}\label{introf}
\prod_{1\le i\le j\le n} \frac{1}{1-x_iy_j}\prod_{1\le i,j\le n} \frac{1}{(qx_iy_j;q)_\infty}\\
%
=
\sum_{\lambda\in (\ZZ_{\ge 0})^n} a_{\lambda}(q) E_\lambda(x;q,0)E_\lambda(y;q^{-1},\infty)
\end{multline}
where $a_{\lambda}(q)$ is the Cherednik norm of $E_\la$. Note that \eqref{introf} reduces to the nonsymmetric Cauchy identity above when $q=0$. One easily sees that the left hand side of
\eqref{introf} computes the character of the space of polynomial functions on the Iwahori-type
subspace of the space of polynomial-valued matrices. 
The main goal of this paper is  
to find the representation theoretic interpretation of formula \eqref{introf} and to write down and prove an analogue of \eqref{introf} for $\fg=\msl_n$.
 We also formulate a conjecture for all simple Lie algebras. Let us describe our results in more details.

In what follows we fix an algebraically closed field $\Bbbk$ of characteristic $0$.
Let $M_n$ be the space of square matrices of size $n$ with coefficients in $\Bbbk$. Let
$M_n[z]=M_n\T \Bbbk[z]$ be the space of matrices whose entries are polynomials in one variable.
We denote by $M_n[z]^+$ the direct sum of the subspace of uppertriangular matrices with the space 
$M_n\T z\Bbbk[z]$ (so $M_n[z]^+$ is the $\mgl_n$ version of the Iwahori subalgebra). Then the left hand
side of \eqref{introf} is naturally identified with the character of the space of polynomial functions
$\Bbbk[M_n[z]^+]$, where $q$ is responsible for the $z$-degree and $x_i$, $y_j$ correspond to the left
and right actions of diagonal torus. In order to obtain the summands in the right hand side we construct 
a filtration $\mathcal{F}_\la$, $\la\in(\bZ_{\ge 0})^n$ on the  restricted dual space $\Bbbk[M_n[z]^+]^\vee$ and describe the homogeneous pieces $\mathcal{F}_\la/\sum_{\mu\succ\la} \mathcal{F}_\mu$. More precisely, 
for a composition $\la$ we consider certain
left module $\mathbb{D}_\lambda$ and a right module $\mathbb{U}_{\lambda}^o$
of the Iwahori algebra of type $\mgl_n$. Both modules are  ($\mgl_n$-versions of) the generalized global Weyl modules with
characteristics (see \cite{FMO,FKM,NS,NNS1,NNS2}). The modules $\mathbb{D}_\lambda$ and  $\mathbb{U}_{\lambda}^o$ are endowed
with the action of the graded (commutative) algebra $\mathcal{A}^D_{\lambda}$ (the so-called 
highest weight algebra) commuting with the action of the current algebra $\mgl_n[z]$ \cite{CFK}). Our first main result is the following:

\begin{trm}
For any $\la\in(\bZ_{\ge 0})^n$,
\[
\mathcal{F}_\la/{\sum_{\mu\succ\la} \mathcal{F}_\mu} \simeq \mathbb{D}_\lambda \otimes_{\mathcal{A}^D_{\lambda} }\mathbb{U}_{\lambda}^o.
\]
\end{trm}

We also show that the algebra $\mathcal{A}^D_{\lambda}$ acts freely on  $\mathbb{D}_\lambda$ and  $\mathbb{U}_{\lambda}^o$. In order to relate Theorem 1 with formula \eqref{introf} we note that (see \cite{FKM,I,OS,RY,Sa})  
\[
\ch\,\mathbb{D}_\lambda = a_{\lambda}(q)E_\lambda(x;q,0),\ \ch\, \mathbb{U}_{\lambda}^o = a_{\lambda}(q)E_\lambda(y;q^{-1},\infty),\
\ch\,\mathcal{A}^D_{\lambda} = a_{\lambda}(q).
\]

The nonsymmetric $q$-Cauchy identity \eqref{introf} has an $\msl_n$-version. We note that the
weight lattice $P$ of $\msl_n$ is obtained as a quotient of the $\mgl_n$-weight lattice $\bZ^n$
by the span of the vector $(1,\dots,1)$. Collecting terms in \eqref{introf} with the same $\msl_n$-weight, we obtain the 
identity
\begin{multline}\label{introsl}
(x_1\dotsm x_n y_1 \dotsm y_n;q)_\infty
\prod_{1\le i\le j\le n} \frac{1}{1-x_iy_j}\prod_{1\le i,j\le n} \frac{1}{(qx_iy_j;q)_\infty}
=\\
\sum_{\lambda\in (\ZZ_{\ge 0})^n_0} 
a_\la(q) E_{\lambda}(x;q,0)E_{\lambda}(y; q^{-1},\infty),
\end{multline}
where $(\ZZ_{\ge 0})^n_0$ is the subset of nonnegative integer vectors with at least one coordinate 
equal to zero. We show that the left hand side coincides with the character of the space of polynomial
functions on the Iwahori group ${\mathbf I}\subset SL_n[z]$. Now our goal is to give a representation
theoretic realization of the right hand side of \eqref{introsl}. To this end we note that  
$(\ZZ_{\ge 0})^n_0$ is identified with the weight lattice $P$ of the Lie algebras $\msl_n$.
For an element $\la\in P$ we consider the $\msl_n[z]$ modules $\mathbb{D}_\lambda$ and 
$\mathbb{U}_{\lambda}^o$, which are special cases of the $\msl_n$ generalized global Weyl modules
with characteristics. Both modules admit an action of a polynomial algebra $\mathcal{A}^D_{\lambda}$ which commutes with the action of the current 
algebra $\msl_n[z]$  (we note that in the simply-laced case the $0$-specialization $\mathbb{D}_\lambda\T_{\mathcal{A}^D_{\lambda}}\Bbbk_0$ is isomorphic to the corresponding level $1$ affine Demazure module).
We introduce an increasing filtration $\mathcal{G}_\la$ on the space of functions 
$\Bbbk[\mathbf{I}]$ such that the following theorem holds:

\begin{trm}\label{Trmsl}
For any $\la\in P$,
\begin{equation}\label{Grsl}
\mathcal{G}_\la/\sum_{\mu\prec\la} \mathcal{G}_\mu \simeq (\mathbb{D}_\lambda \otimes_{\mathcal{A}^D_{\lambda} }\mathbb{U}_{\lambda}^o)^\vee,
\end{equation}
where  $\vee$ in the upper index denotes the restricted dual space. 
\end{trm}
As in the $\mgl_n$ case 
 algebra $\mathcal{A}^D_{\lambda}$ acts freely on  $\mathbb{D}_\lambda$ and  $\mathbb{U}_{\lambda}^o$. There is a natural map from the right-hand side to the left-hand side of \eqref{Grsl}. The numerical identity \eqref{introsl} is equivalent to the assertion that this map is an isomorphism.
%

The filtration $\mathcal{G}_\la$ in the $\msl_n$ case can be seen as the generalization of the van der Kallen
result \cite{vdK} to the Iwahori case. We conjecture that Theorem~\ref{Trmsl} holds true for arbitrary simple Lie algebra
(at the moment we are unable to prove it due to several technical difficulties). Let $P$ be a weight lattice of a
simple Lie algebra $\fg$ and let $\Delta=\Delta_+\sqcup\Delta_-$ be the set of roots of $\fg$. 
We also let ${\mathbf I}$ denote
the Iwahori subgroup inside $G[z]$, where $G$ is the Lie group of $\fg$.

\begin{con}\label{Conj1}
The space of functions $\Bbbk[\mathbf{I}]$ admits a filtration such that the associated graded space is isomorphic
to the direct sum over $\la\in P$ of modules of the form $(\mathbb{D}_\lambda \otimes_{\mathcal{A}^D_{\lambda} }\mathbb{U}_{\lambda}^o)^\vee$ for certain generalized global Weyl modules with characteristics $\mathbb{D}_\lambda$
and $\mathbb{U}_{\lambda}^o$.    
\end{con}

We are able to show that for all $\fg$ the modules $\mathbb{D}_\lambda$ and $\mathbb{U}_{\lambda}^o$  admit a free action of  the highest weight algebra $\mathcal{A}^D_{\lambda}$.  The combinatorial counterpart of Conjecture \ref{Conj1} is given as follows.

\begin{con}\label{Conj2}
One has an identity
\begin{equation*}
\frac{\sum_{\la\in P} X^\la Y^\la}{\prod_{\al\in\Delta_+} (1-X^\al)\prod_{\al\in\Delta} (qX^\al;q)_\infty}=\\
\sum_{\lambda\in P} 
a_\la(q) E_{\lambda}(X;q,0)E_{\lambda}(Y; q^{-1},\infty)
\end{equation*}
for $a_\la(q)$ being the Cherednik norm of $E_\la$.
\end{con}

For $\msl_n$, Conjecture~\ref{Conj2} follows from \eqref{introsl}, as we will show in Theorem~\ref{ThmSlNonsymmetricCauchy}.
We note that a closely related question to Conjecture~\ref{Conj1} is the existence of a BGG reciprocity theorem in the Iwahori setting (see \cite{CI,CG,Kh}). We plan to address this question elsewhere.

Our paper is organized as follows. In Section \ref{Pre} we collect basic notation on representations of simple Lie algebras.  In Section \ref{Nonsymcomb}  we study the combinatorics of the nonsymmetric $q$-Cauchy identities for $\mgl_n$ and $\msl_n$. 
In Section \ref{Current} we introduce and study certain families of modules of the Iwahori algebra. These modules are used in Section  \ref{glfiltr} to derive a representation theoretic description of the nomsymmetric $q$-Cauchy identity and in Section \ref{slfunctions} to prove an analogous theorem in the $\msl_n$ case. Finally, in Appendix we consider the rank one example. 

\subsection*{Acknowledgements}

We thank Ivan Cherednik and Satoshi Naito for helpful conversations. D.O. gratefully acknowledges support from the Simons Foundation (Collaboration Grant for Mathematicians, 638577) and the Max Planck Institute for Mathematics (MPIM Bonn).

\section{Preliminaries}\label{Pre}

\subsection{Simple Lie algebras}\label{RootData}
Suppose $\fg$ is a simple Lie algebra over $\bC$ with Cartan decomposition $\fg=\fn_+\oplus\fh\oplus\fn_-$. Let
$\Delta=\Delta_+\sqcup \Delta_-\subset\fh^*$ be the root system of $\fg$. For each $\al\in\Delta_+$, we denote by $e_\al\in\fn_+$ the corresponding Chevalley generator of $\g$. Similarly, for
$\al\in\Delta_-$, we denote by $f_\al$ the Chevalley generator of weight $\al$ in $\fn_-$.
For $\alpha\in \Delta$, we write $\alpha>0$ for $\alpha\in\Delta_+$ and $\alpha<0$ for $\alpha\in\Delta_-$. We denote by $\{ \al_i \}_{i \in I}$ the simple roots and by
$\{\om_i\}_{i \in I}$ the fundamental weights. For a root $\al\in\Delta$, we denote by $\al^\vee\in\fh$ the corresponding coroot.
For the canonical pairing $\bra\cdot,\cdot\ket:\fh^*\times\fh\to\bC$, we have $\bra \om_i,\al_j^\vee \ket
=\delta_{ij}$. 

Let $Q=\bigoplus_{i\in I}\ZZ\alpha_i$ be the root lattice and $Q_+=\bigoplus_{i\in I}\ZZ_{\ge 0}\alpha_i$ the positive root cone. Let $P=\bigoplus_{i\in I}\bZ\om_i$ be the weight lattice and $P_+=\sum_{i\in I} \bZ_{\ge 0} \om_i$ the dominant weight cone. Let $P_- := - P_+$ be the antidominant weight cone. Let $\rho=\frac{1}{2}\sum_{\alpha\in\Delta_+} \alpha\in P_+$.

Let $W$ be the Weyl group of $\fg$. For a root $\al\in\Delta$, the corresponding reflection is denoted by $s_\al\in W$. For each $i \in I$, we set $s_{i} := s_{\al_{i}}$. Let $w_0$ be the longest element in $W$.

For each $\lambda \in P$, let $\lambda_-\in P_-$ be the unique antidominant weight belonging to the $W$-orbit of $\lambda$. Denote by $v(\lambda)$ the unique element of $W$ of minimal length such that $v(\lambda)\lambda=\lambda_-$. Define a partial order $\succeq$ on $P$ by $\lambda \succeq \mu$ if and only if $\la_--\mu_-\in -Q_+$ and $v(\lambda)\leq v(\mu)$ when $\la_-=\mu_-$.

\subsection{Characters}
Let $\ZZ[P]$ be the group algebra of $P$, which is a free $\ZZ$-module with basis $\{X^\lambda\}_{\lambda\in P}$ satisfying $X^\lambda X^\mu=X^{\lambda+\mu}$ for all $\lambda,\mu\in P$. Denote by $\ZZ[[P]]$ be the $\ZZ[P]$-module consisting of all (possibly infinite) formal sums $\sum_{\lambda\in P} c_\lambda X^\lambda$ where $c_\lambda\in\ZZ$. Setting $X_i=X^{\omega_i}$ for $i\in I$, we can also write $X^\lambda=\prod_{i\in I}X_i^{\langle\lambda,\alpha_i^\vee\rangle}$.

For any (left) $\fh$-module $M$ and $\lambda\in\fh^*$, we set $${}_\lambda M=\{m\in M : \text{$hm=\lambda(h)m$ for all $h\in \fh$}\}.$$ If $M$ has a decomposition
\[M=\bigoplus_{\lambda \in P}{}_\lambda M,\qquad \dim {}_\lambda M<\infty, \]
then we define the character of $M$ as the formal sum
\[\ch\, M=\sum_{\lambda \in P}(\dim {}_\lambda M) X^\lambda\in\ZZ[[P]].\]

Analogously, for any right $\fh$-module $M$, we set $$M_\lambda=\{m\in M : \text{$mh=\lambda(h)m$ for all $h\in \fh$}\}.$$ If $M$ has a decomposition
\[M=\bigoplus_{\lambda \in P}M_\lambda, \qquad \dim M_\lambda <\infty,\]
then we define the right character
\[\ch\, M=\sum_{\lambda \in P}(\dim M_\lambda) Y^\lambda\]
inside another copy of $\ZZ[[P]]$ with basis elements denoted $\{Y^\lambda\}_{\lambda\in P}$.

Finally, for $M$ an $(\fh,\fh)$-bimodule, we set ${}_\lambda M_\mu={}_\lambda M \cap M_\mu$ for all $\lambda,\mu\in \fh^*$. If $M$ has a decomposition
\[M=\bigoplus_{\lambda,\mu\in P} {}_\lambda M_\mu, \dim {}_\lambda M_\mu<\infty,\]
then we define the bimodule character as
\[\ch\, M=\sum_{\lambda,\mu \in P}(\dim {}_\lambda M_\mu) X^\lambda Y^\mu\]
as an element of $\ZZ[[P\times P]]=\{\sum_{\lambda,\mu} c_{\lambda,\mu} X^\lambda Y^\mu : c_{\lambda,\mu}\in\ZZ\}$. The latter is a module over $\ZZ[P\times P]\cong\ZZ[P]\otimes\ZZ[P]$, whose elements are finite sums of the above form. 

In the graded setting (see \S\ref{curalg}), we also consider the abelian group $(\ZZ[[P\times P]])[[q]]$ consisting of power series in $q$ with coefficients in $\ZZ[[P\times P]]$, which is a module over the ring $(\ZZ[P\times P])[[q]]$.

\subsection{$\gl_n$ and $\msl_n$}\label{GlNotation}

We also consider the Lie algebra $\gl_n$. Let $E_{ij}\ (i,j=1, \dots, n)$ be the matrix units. We consider the Cartan subalgebra $\fh \subset \mgl_n$ of diagonal matrices, $\fh=\mathrm{span}\{E_{11}, \dots, E_{nn}\}$. We use the standard Cartan decomposition with $\fn_+=\mathrm{span}\{E_{ij}\}_{i<j}.$
We denote by $P(\gl_n)$ the weight lattice $\bigoplus_{i=1}^n \ZZ\varepsilon_i$. The roots are the elements $\varepsilon_i-\varepsilon_j$, $i\neq j$, with the positive roots given by $i<j$. We use the notation $h_i:=E_{ii}$ and $e_\alpha=E_{ij}$ and $f_{-\alpha}=E_{ji}$ for any positive root $\alpha=\varepsilon_i-\varepsilon_j$, $i<j$.

We identify $\lambda=(\lambda_1, \dots, \lambda_n) \in \ZZ^n$ with the element $\sum_{i=1}^n \lambda_i \varepsilon_i\in P(\gl_n)$. Dominant weights then correspond to tuples $(\lambda_1,\dotsc,\lambda_n)\in\ZZ^n$ such that $\lambda_1\ge \dotsm\ge\lambda_n$. The Weyl group is the symmetric group $S_n$.

Elements of $(\ZZ_{\ge 0})^n$ are called compositions. In the $\gl_n$ case, we consider characters of $\fh$-modules whose weights are compositions. Thus, for a left $\fh$-module $M$ such that 
\[M=\bigoplus_{\lambda \in (\ZZ_{\ge 0})^n}{}_\lambda M, \qquad \dim {}_\lambda M<\infty,\]
where $M_\lambda=\{m\in M : \text{$h_i m = \lambda_i m$ for $i=1,\dotsc,n$}\}$,
we define the character
\[\ch\, M=\sum_{\lambda \in P}(\dim {}_\lambda M) x^\lambda\in\ZZ[[x_1,\dotsc,x_n]]\]
where $x^\lambda=x_1^{\lambda_1}\dotsm x_n^{\lambda_n}$. The characters of right modules are defined in a similar way with variables $\{y_i\}$ instead of $\{x_i\}$. Characters of bimodules are then defined using both sets of variables in the natural way, as above.

We realize the weight lattice of $\msl_n$ as the quotient $P(\msl_n)=P(\gl_n)/\ZZ\one$, where $\one=\varepsilon_1+\dotsm+\varepsilon_n$, and we denote the quotient map $P(\gl_n)\to P(\msl_n)$ by $\lambda\mapsto\lab$. Each representation of the Cartan subalgebra, Borel subalgebra etc. in type $\mgl_n$ can be considered as representations of the corresponding subalgebras in type $\msl_n$. The restriction map on characters of finite-dimensional modules is achieved by the surjective ring homomorphism
\begin{align*}
\ZZ[x_1, \dots, x_n] &\overset{\res}{\longrightarrow} \ZZ[P(\msl_n)]\\
x_i &\longmapsto X^{\veb_i}
\end{align*}
Recalling the notation $X_i=X^{\omega_i}$, where $\omega_i \ (i=1,\dotsc,n-1)$ are the fundamental weights of $\msl_n$, we can also describe this map as follows:
\begin{align*}
\res(x_1)=X_1,\quad \res(x_i)=X_{i-1}^{-1}X_i (i=2, \dots, n-1),\quad \res(x_n) \mapsto X_{n-1}^{-1}.  
\end{align*}
Note that $\ZZ[P(\msl_n)]\cong \ZZ[X_1^{\pm 1},\dotsm,X_{n-1}^{\pm 1}]$ is a Laurent polynomial ring, whereas for $\mgl_n$ we have an ordinary polynomial ring because we only consider representations whose weights are compositions.

\subsection{$q$-Pochammer symbols}

Define $(a;q)_m=\prod_{i=0}^{m-1}(1-q^ia)$ for $m\ge 0$, $(a;q)_\infty=\prod_{i=0}^\infty(1-q^ia)$, $(q)_m=(q;q)_m$, and $(q)_\infty=(q;q)_\infty$.

\section{Nonsymmetric $q$-Cauchy identities in type $A$}\label{Nonsymcomb}

We recall the Cauchy identity of \cite{MN} for $\gl_n$ nonsymmetric Macdonald polynomials and consider its $t=0$ specialization. We establish a corresponding identity for specialized $\msl_n$ nonsymmetric Macdonald polynomials. Note that we use the conventions of \S\ref{GlNotation} throughout this section.

\subsection{Macdonald polynomials}

Fix a positive integer $n$ and, as in \S\ref{GlNotation}, let $x=(x_1,\dotsc,x_n)$ and $y=(y_1,\dotsc,y_n)$ be variables. 
For any composition $\lambda\in (\ZZ_{\ge 0})^n$, denote by $E_\lambda=E_\lambda(x;q,t)\in\K[x]$ the $\gl_n$ nonsymmetric Macdonald polynomial as defined in \cite{Ch2,HHL}, where $\K=\QQ(q,t)$. The polynomial $E_\lambda(x;q,t)$ is homogeneous in $x_1,\dotsc,x_n$ of total degree $|\lambda|:=\lambda_1+\dotsm+\lambda_n$.

The diagram of $\lambda\in (\ZZ_{\ge 0})^n$ is the set $$\dg(\lambda)=\{(i,j) : 1\le i\le n, 1\le j\le \lambda_i\}.$$ We write $(i,j)\in\lambda$ to mean $(i,j)\in\dg(\lambda)$. Following the conventions of \cite{HHL}, we have we have the generalized arm and leg functions for $s=(i,j)\in\lambda$:
\begin{align*}
\leg_\lambda(s) &= \lambda_i-j\\
\arm_\lambda(s) &= |\{k<i : j\le \lambda_k\le \lambda_i \}|+|\{k>i : j\le \lambda_k+1\le \lambda_i\}|.
\end{align*}
We define
\begin{align*}
a_\lambda(q,t) &= \prod_{s\in\lambda}\frac{1-q^{\leg_\lambda(s)+1}t^{\arm_\lambda(s)+1}}{1-q^{\leg_\lambda(s)+1}t^{\arm_\lambda(s)}}.
\end{align*}

The $\msl_n$ nonsymmetric Macdonald polynomials $E_\lambda(X;q,t)\in\K[P(\msl_n)]$ (for $\lambda\in P(\msl_n)$) of \cite{Ch2} are related to those for $\mgl_n$ via the restriction map $\res$ of \S\ref{GlNotation} (extended $\K$-linearly). Specifically, for any $\lambda\in(\ZZ_{\ge 0})^n\subset P(\gl_n)$, one verifies directly from the definitions of these polynomials that
\begin{align}\label{ResE}
\res(E_\lambda(x;q,t)) = E_{\lab}(X;q,t).
\end{align}

\subsection{Nonsymmetric Cauchy identity}

The nonsymmetric Cauchy identity of \cite{MN} gives the expansion of the function
\begin{align*}
E(x,y;q,t)
&=
\prod_{i=1}^n \frac{1}{1-x_iy_i}\prod_{1\le i<j\le n}\frac{1-t x_iy_j}{1-x_iy_j}
\prod_{1\le i, j\le n} \frac{(qtx_iy_j;q)_\infty}{(qx_iy_j;q)_\infty}.
\end{align*}
into products of nonsymmetric Macdonald polynomials.

\begin{thm}[{\cite[Theorem 2.1]{MN}}]
One has
\begin{align}\label{NonsymmetricCauchy-qt}
E(x,y;q,t)
&=
\sum_{\lambda\in (\ZZ_{\ge 0})^n} a_{\lambda}(q,t)E_\lambda(x; q,t)E_\lambda(y; q^{-1},t^{-1}).
\end{align}
\end{thm}

\begin{proof}
Strictly speaking, \cite[Theorem 2.1]{MN} asserts that
\begin{align*}
E(x,y;q,t)
&=
\sum_{\lambda\in (\ZZ_{\ge 0})^n} \tia_{\lambda}(q,t)\tE_\lambda(x; q,t)\tE_\lambda(y; q^{-1},t^{-1}),
\end{align*}
where
\begin{align*}
\tia_\lambda(q,t)&=a_{w_0(\lambda)}(q,t)\\
\tE_\lambda(x;q,t)&=w_0 E_{w_0(\lambda)}(x;q^{-1},t^{-1}).
\end{align*}
It is straightforward to see that the two formulas are equivalent.
\end{proof}


We regard \eqref{NonsymmetricCauchy-qt} as an equality meromorphic functions. Both sides are regular for $|q|$, $|t|$, $|x_i|$, $|y_i|<1$. It is not immediate but well-known (see, e.g., \cite[Eqn. (37)]{HHL}) that $E_\lambda(y; q^{-1},t^{-1})$ is regular at $(q,t)=(0,0)$.

\subsection{$t=0$ specialization}
Setting $t=0$ in \eqref{NonsymmetricCauchy-qt}, we obtain:

\begin{cor} We have
\begin{align}\label{NonsymmetricCauchy}
E(x,y;q)
%
&=
\sum_{\lambda\in (\ZZ_{\ge 0})^n} a_{\lambda}(q) E_\lambda(x;q,0)E_\lambda(y;q^{-1},\infty)
\end{align}
where
\begin{align*}
E(x,y;q)&=E(x,y;q,0)
=\prod_{1\le i\le j\le n} \frac{1}{1-x_iy_j}\prod_{1\le i,j\le n} \frac{1}{(qx_iy_j;q)_\infty}
\end{align*}
and
\begin{align*}
a_\lambda(q) &=a_\lambda(q,0)= \prod_{\substack{s\in\lambda\\\arm_\lambda(s)=0}}\frac{1}{1-q^{\leg_\lambda(s)+1}}.
\end{align*}
\end{cor}

It follows from, e.g., \cite{HHL} that both of $E_\lambda(x;q,0)$ and $E_\lambda(y;q^{-1},\infty)$ have coefficients in $\ZZ_{\ge 0}[q]$. Moreover,  by the homogeneity of $E_\lambda$, one sees that the number of summands on the right-hand side of \eqref{NonsymmetricCauchy} contributing to any fixed monomial $x^\mu y^\nu$ is finite. In particular, both sides are convergent as formal power series and we may regard \eqref{NonsymmetricCauchy} as an identity in $\ZZ[[x,y,q]]=\ZZ[[x_1,\dotsc,x_n,y_1,\dotsc,y_n,q]]$.


The following lemma gives another expression for the quantity $a_\lambda(q)$:

\begin{lem}
For any $\lambda\in(\ZZ_{\ge 0})^n$, we have:
\begin{align}\label{aAlt}
a_{\lambda}(q) &= (q;q)_{(\lambda_-)_1}^{-1}\cdot \prod_{j=1}^{n-1}
\begin{cases}
(q;q)_{-\langle\la_-,\al_j^\vee\rangle}^{-1} &\text{if $v(\lambda)^{-1}\alpha_j>0$ }\\
(q;q)_{-\langle\la_-,\al_j^\vee\rangle-1}^{-1} &\text{if $v(\lambda)^{-1}\alpha_j<0$}
\end{cases}
\end{align}
where $\lambda_-\in(\ZZ_{\ge 0})^n$ and $v(\lambda)\in S_n$ are as in \S\ref{RootData}.
\end{lem}

\begin{proof}
By \cite[Lemma~2.5]{MN}, we have the following for $\lambda\in (\ZZ_{\ge 0})^n$:
\begin{align*}
a_\lambda(q,t) &= 
a_{\lambda_-}(q,t)
\prod_{\substack{\alpha>0\\\langle\lambda,\alpha^\vee\rangle>0}}
\frac{(1-q^{\langle\lambda,\alpha^\vee\rangle}t^{-\langle\rho,v(\lambda)\alpha^\vee\rangle+1})(1-q^{\langle\lambda,\alpha^\vee\rangle} t^{-\langle\rho,v(\lambda)\alpha^\vee\rangle-1})}{(1-q^{\langle\lambda,\alpha^\vee\rangle}t^{-\langle \rho,v(\lambda)\alpha^\vee\rangle})^2}
\end{align*}
Note that $\langle\lambda,\alpha^\vee\rangle>0$ implies $\langle\lambda_-,v(\lambda)\alpha^\vee\rangle>0$ and hence $v(\lambda)\alpha<0$. Therefore, we may set $t=0$ in the expression above to obtain
\begin{align*}
a_\lambda(q,0) &= 
a_{\lambda_-}(q,0)
\prod_{\substack{\alpha>0\\\langle\lambda,\alpha^\vee\rangle>0\\\langle\rho,v(\lambda)\alpha^\vee\rangle=-1}}(1-q^{\langle\lambda,\alpha^\vee\rangle}).
\end{align*}

Now, it is easy to see that the map $\alpha\mapsto -v(\lambda)\alpha$ gives a bijection
between the sets
$\{\alpha>0 : \text{$\langle\lambda,\alpha^\vee\rangle>0$ and $\langle\rho,v(\lambda)\alpha^\vee\rangle=-1$}\}$ and
$\{ \alpha_j : v(\lambda)^{-1}\alpha_j<0\}$.
Moreover, for $\alpha>0$ and $\alpha_j=-v(\lambda)\alpha$ belonging to these sets, we have $\langle\lambda,\alpha^\vee \rangle = -\langle\lambda_-,\alpha_j^\vee\rangle.$ Hence we can write
\begin{align*}
a_\lambda(q) &= 
a_{\lambda_-}(q)
\prod_{\substack{1\le j\le n-1\\v(\lambda)^{-1}\alpha_j<0}}(1-q^{-\langle\lambda_-,\alpha_j^\vee\rangle}).
\end{align*}

Finally, we consider
\begin{align*}
a_{\lambda_-}(q) = 
\prod_{\substack{
s\in \lambda_-\\
\arm_{\lambda_-}(s)=0}} \frac{1}{1-q^{\leg_{\lambda_-}(s)+1}}
\end{align*}
The cells $s\in\lambda_-$ with $\arm_{\lambda_-}(s)=0$ are precisely the following: $(j+1,(\lambda_-)_{j}+1),\dotsc,(j+1,(\lambda_-)_{j+1})$ for $1\le j<n$ and $(1,1),\dotsc,(1,(\lambda_-)_1)$. These cells have $\leg_{\lambda_-}(s)+1$ equal to $(\lambda_-)_{j+1}-(\lambda_-)_{j},\dotsc,1$ for $1\le j<n$ and $(\lambda_-)_1,\dotsc,1$, respectively. Combining this with the preceding paragraph, we obtain the stated formula.
\end{proof}




\subsection{From $\gl_n$ to $\msl_n$}

Now, by considering the behavior of \eqref{NonsymmetricCauchy} under the projection $P(\gl_n)\to P(\msl_n)$, we shall prove the following:

\begin{thm}\label{ThmSlNonsymmetricCauchy}
Let $E_\lambda(X;q,t)$ be the $\msl_n$ nonsymmetric Macdonald polynomials. Then we have the following identity in $(\ZZ[[P(\msl_n)\times P(\msl_n)]])[[q]]$:
\begin{multline}
\label{SlNonsymmetricCauchy}
\frac{\sum_{\lambda\in P(\msl_n)}X^\lambda Y^\lambda}{(q;q)_\infty^{n-1}}
\prod_{1\le i<j\le n} \frac{1}{1-X^{\veb_i}Y^{\veb_j}}\prod_{1\le i\neq j\le n} \frac{1}{(qX^{\veb_i}Y^{\veb_j};q)_\infty}\\
=\sum_{\lambda\in P(\msl_n)} 
\ch(\cA^{D}_{\lambda}) E_{\lambda}(X;q,0)E_{\lambda}(Y;q^{-1},\infty).
\end{multline}
where $\cA^D_\lambda$ is the graded polynomial algebra defined in \eqref{AD} below and $\ch(\cA^{D}_{\lambda})\in\ZZ[[q]]$ is its Hilbert series.
\end{thm}

\begin{rem}
Even though $(\ZZ[[P(\msl_n)\times P(\msl_n)]])[[q]]$ is not a ring, our proof will show that the product on the left-hand side of \eqref{SlNonsymmetricCauchy} is well-defined (when expanded formally).
\end{rem}

\begin{proof}
We begin by eliminating some redundancy on the right-hand side of \eqref{NonsymmetricCauchy}. Let $(\ZZ_{\ge 0})^n_0\subset(\ZZ_{\ge 0})^n$ be the subset of nonnegative integer vectors with at least one coordinate equal to $0$. Then every element of $(\ZZ_{\ge 0})^n$ can be written uniquely in the form $\lambda+m\one$ for some $\lambda\in(\ZZ_{\ge 0})^n_0$ and $m\ge 0$. Moreover, each element of $P(\msl_n)$ has a unique representative from $(\ZZ_{\ge 0})^n_0$.

It is immediate from the definition of $E_\lambda$ (e.g., \cite[Def. 2.1.1]{HHL}) that
\begin{align*}
E_{\lambda+m\one}(x;q,t) &= (x_1\dotsm x_n)^m E_\lambda(x;q,t).
\end{align*}
Thus we can write the right-hand side of \eqref{NonsymmetricCauchy} as
\begin{align*}
\sum_{\lambda\in (\ZZ_{\ge 0})^n_0} a_\lambda(q) E_{\lambda}(x;q,0)E_{\lambda}(y; q^{-1},\infty)\sum_{m=0}^\infty\frac{(x_1\dotsm x_n y_1 \dotsm y_n)^m}{(q;q)_m}.
\end{align*}
By the $q$-binomial theorem, we can simplify the inner sum as follows:
\begin{align*}
\sum_{m=0}^\infty\frac{(x_1\dotsm x_n y_1 \dotsm y_n)^m}{(q;q)_m} = \frac{1}{(x_1\dotsm x_n y_1 \dotsm y_n;q)_\infty}.
\end{align*}
Moving this factor to the other side of \eqref{NonsymmetricCauchy}, we arrive at
\begin{multline}\label{SlIwahoriCharacter}
(x_1\dotsm x_n y_1 \dotsm y_n;q)_\infty
\prod_{1\le i\le j\le n} \frac{1}{1-x_iy_j}\prod_{1\le i,j\le n} \frac{1}{(qx_iy_j;q)_\infty}\\
=
\sum_{\lambda\in (\ZZ_{\ge 0})^n_0} 
a_\lambda(q) E_{\lambda}(x;q,0)E_{\lambda}(y; q^{-1},\infty).
\end{multline}

The basic idea now is that we would like to consider the image of \eqref{SlIwahoriCharacter} under an extension of the projection $\res\otimes \res: x^\lambda y^\mu\mapsto X^{\lab}Y^{\mub}$ to infinite sums. However, this is not always defined and some care is required. For $F(x,y) = \sum_{\lambda,\mu\in(\ZZ_{\ge 0})^n} c_{\lambda,\mu} x^\lambda y^\mu \in \ZZ[[x,y]]$, the sum
\begin{align*}
\Fb(X,Y) &= \sum_{\lambda,\mu\in(\ZZ_{\ge 0})^n} c_{\lab,\mub} X^{\lab}Y^{\mub}\in \ZZ[[P(\msl_n)\times P(\msl_n)]]
\end{align*}
is well-defined \textit{provided} that for each $\lambda,\mu\in(\ZZ_{\ge 0})^n$ there exist finitely many integers $k,l$ such that $c_{\lambda+k\one,\mu+l\one}\neq 0$.
Thus, for 
$$F(x,y,q) = \sum_{k=0}^\infty F_k(x,y)q^k\in\ZZ[[x,y,q]]=(\ZZ[[x,y]])[[q]],$$ we define 
\begin{align*}
\Fb(X,Y,q)=\sum_{k=0}^\infty \Fb_k(x,y)q^k\in (\ZZ[[P(\msl_n)\times P(\msl_n)]])[[q]]
\end{align*}
precisely when each $F_k(x,y)$ satisfies this finiteness condition.

We now apply these considerations to \eqref{SlIwahoriCharacter}, beginning with the left-hand side.

\begin{lem}
For
\begin{align*}
F(x,y)=\frac{1-x_1\dotsm x_n y_1\dotsm y_n}{\prod_{i=1}^n (1-x_iy_i)}\prod_{1\le i<j\le n}\frac{1}{1-x_iy_j}
\end{align*}
the $\msl_n\times\msl_n$ character $\Fb(X,Y)$ is well-defined and equal to
\begin{align*}
\Fb(X,Y) = \sum_{\lambda\in P(\msl_n)}X^\lambda Y^\lambda
\cdot \prod_{1\le i<j\le n} \frac{1}{1-X^{\veb_i}Y^{\veb_j}}.
\end{align*}
\end{lem}

\begin{proof}

First, we have the following in $\ZZ[[x,y]]$, where we set $z_i=x_iy_i$:
\begin{align*}
\frac{1-x_1\dotsm x_n y_1 \dotsm y_n}{\prod_{i=1}^n (1-x_iy_i)}
&= (1-z_1\dotsm z_n)\sum_{\nu\in(\ZZ_{\ge 0})^n} z_1^{\nu_1}\dotsm z_n^{\nu_n}\\
&= \sum_{\nu\in(\ZZ_{\ge 0})^n_0} z_1^{\nu_1}\dotsm z_n^{\nu_n}.
\end{align*}
Thus we can write
\begin{align}\label{FExpand}
F(x,y)=\sum_{\nu\in(\ZZ_{\ge 0})^n_0} \sum_{\{m_{i,j}\}} x^{\nu+\sum_{i<j}m_{i,j}\varepsilon_i}y^{\nu+\sum_{i<j}m_{i,j}\varepsilon_j}
\end{align}
where the inner sum runs over all collections $\{m_{i,j}\}_{i<j}$ of nonnegative integers. We need to show that for any fixed $\nu$ and $\{m_{i,j}\}$ as above, there exist finitely many other $\nu'$ and $\{m'_{i,j}\}$ such that
\begin{align}
\label{GltoSlEqualTerms}
\nu+\sum_{i<j}m_{i,j}\varepsilon_i+k(1,\dotsc,1)=\nu'+\sum_{i<j}m_{i,j}'\varepsilon_i\\
\nu+\sum_{i<j}m_{i,j}\varepsilon_j+l(1,\dotsc,1)=\nu'+\sum_{i<j}m_{i,j}'\varepsilon_j\notag
\end{align}
for some $k,l\in\ZZ$. Taking the sum of all components on both sides, we obtain
\begin{align}
\label{GltoSlSumCoords}
|\nu|+\sum_{i<j}m_{i,j}+kn=|\nu'|+\sum_{i<j}m_{i,j}'\\
|\nu|+\sum_{i<j}m_{i,j}+ln=|\nu'|+\sum_{i<j}m_{i,j}'\notag
\end{align}
and these together imply that $k=l$. Now taking the difference of our original equations, we obtain
\begin{align}\label{Kostant}
\sum_{i<j}m_{i,j}(\varepsilon_i-\varepsilon_j)=\sum_{i<j}m_{i,j}'(\varepsilon_i-\varepsilon_j).
\end{align}
For fixed $\{m_{i,j}\}$, the number of $\{m_{i,j}'\}$ satisfying \eqref{Kostant} is finite. Returning to \eqref{GltoSlEqualTerms}, we now see that $k$ is bounded above, since the smallest coordinate of the right-hand side is constrained by the requirement that $\nu'\in(\ZZ_{\ge 0})^n_0$ and the fact that only finitely many $\{m_{i,j}'\}$ are possible. Of course, $k$ is also bounded below, because the left-hand side of \eqref{GltoSlSumCoords} must be nonnegative. We conclude that there are only finitely many possibilities for $k$ and therefore only finitely many $\nu'$.

Knowing now that $\Fb(X,Y)$ is well-defined, the stated formula follows from \eqref{FExpand}.
\end{proof}

The remaining factors on the left-hand side \eqref{SlIwahoriCharacter} actually belong to the subring $\ZZ[x,y][[q]]\subset \ZZ[[x,y,q]]$ and hence map to elements of the ring $\ZZ[P(\msl_n)\times P(\msl_n)][[q]]$. From this one deduces that the left-hand side of \eqref{SlIwahoriCharacter} has well-defined projection to $\ZZ[[P(\msl_n)\times P(\msl_n)]][[q]]$ and that the result is given by the left-hand side of \eqref{SlNonsymmetricCauchy}. Here one must observe that
$$ (qX^\mu Y^\mu;q)_\infty^{\pm 1} \sum_{\lambda\in P(\msl_n)} X^\lambda Y^\lambda = (q;q)_\infty^{\pm 1} \sum_{\lambda\in P(\msl_n)} X^\lambda Y^\lambda$$
for any $\mu\in P(\msl_n)$.

Finally, since the right-hand side of \eqref{SlIwahoriCharacter} is a sum of nonnegative terms its projection must also be well-defined in the sense above, and we can project each summand individually to arrive at \eqref{SlNonsymmetricCauchy}. We use \eqref{ResE} to go from $\gl_n$ to $\msl_n$ Macdonald polynomials and then \eqref{aAlt} to identify $a_\lambda(q)$ for $\lambda\in(\ZZ_{\ge 0})^n$ with the graded character of $\cA^D_{\lab}$ given by \eqref{AD}.
\end{proof}


\section{Representations of current algebras}\label{Current}

\subsection{Current algebras} \label{curalg}
Let $\fg[z]=\fg\T \bC[z]$ be the current algebra of a simple Lie algebra $\fg$. We have a grading on $\fg [z]$ by setting $\deg \, a \otimes z^m = m$ for each $a \in \g \setminus \{ 0 \}$ and $m \ge 0$. We use the notation $a z^m =a \otimes z^m$, $a=a \otimes 1$. 

Let $\mathcal{I} = \fn^+\oplus\fh \oplus\fg\T z\bC[z]\subset\fg\T\bC[z]$ be the Iwahori subalgebra.
An $\mathcal{I}$ module $M$ is called graded, if $M=\bigoplus_{j\in \bZ} M_j$ such that each $M_j$ is $\fh$ semi-simple
(i.e. each $M_j$ is the sum of $\fh$ weight spaces) and $(g\T z^i) M_j\subset M_{i+j}$ for all $g\in\fg$,
$i\ge 0$. We define the character of $M$ as the formal linear combination
\[
\ch \, M=\sum_{j\in\bZ} q^j\ch \, M_j,
\]
there $\ch \, M_j$ is the $\fh$-module character. In what follows we always consider the modules $M$ whose
$\fh$-weights belong to $P$. We say that $\ch \, M$ is well-defined whenever we have $\ch \, M_j \in \bZ [P]$
for each $j \in \bZ$. If $M$ is cyclic with cyclic vector $v$, then we assume that
$v\in M_0$ unless stated otherwise.

Let $M_j=\bigoplus_{\la\in P} M_{j,\la}$ be the $\fh$-weights decomposition with $\dim M_{j,\la}<\infty$. 
We denote by $M^\vee$ the restricted dual space $M^\vee=\bigoplus_{j\in\bZ}\bigoplus_{\la\in P} M_{j,\la}^*$. 

Recall the $W$-action on the set
$\{e_\al\}_{\al\in\Delta_+}\cup\{f_{-\al}\T z\}_{\al\in\Delta_+}$ following \cite{FeMa3}: for an element
$\sigma\in W$
and $\al\in\Delta_+$ we set
\[
\widehat{\sigma} e_\al=\begin{cases}
e_{\sigma(\al)}, \qquad  \sigma(\al)\in\Delta_+,\\
f_{\sigma(\al)}\T z, \ \sigma(\al)\in\Delta_-,
\end{cases}
\widehat{\sigma} (f_{-\al}\T z)=\begin{cases}
e_{-\sigma(\al)}, \qquad  \sigma(\al)\in\Delta_-,\\
f_{-\sigma(\al)}\T z, \ \sigma(\al)\in\Delta_+.
\end{cases}
\]
We will also use the following notation for $\al\in\Delta_+$ and $r\ge 0$:
\begin{gather*}
e_{\widehat{\sigma}\al+r\delta}=\begin{cases}
e_{\sigma(\al)}\T z^r, &  \sigma(\al)\in\Delta_+,\\
f_{\sigma(\al)}\T z^{r+1}, & \sigma(\al)\in\Delta_-,
\end{cases},\\
e_{\widehat{\sigma}(-\al+\delta)+r\delta}=\begin{cases}
e_{-\sigma(\al)}\T z^r, &  \sigma(\al)\in\Delta_-,\\
f_{-\sigma(\al)}\T z^{r+1}, & \sigma(\al)\in\Delta_+.
\end{cases}
\end{gather*}

Let $\la_-\in P_-$, $\sigma\in W$ and let $\lambda=\sigma(\lambda_-)\in P$.
\begin{dfn}
The generalized Weyl module $W_{\lambda}$ is the cyclic $\mathcal{I}$ module with the cyclic vector $w_\lambda$ defined by the following relations for all $\al\in\Delta_+$:
\begin{align*}
 \widehat{\sigma} (f_{-\al}\T z) z^k w_\lambda=0,\ k\ge 0;\\
 ( \widehat{\sigma} (e_{\al}))^{\langle -\lambda_-, \al^\vee \rangle+1}  w_\lambda=0;\\
 z\fh[z]w_\lambda=0.
\end{align*}
The definition of global generalized Weyl module  $\mathbb{W}_{\lambda}$ is obtained by omitting the last relation.
\end{dfn}
\begin{rem}
If $\sigma$ is the identity element (i.e. $\la\in P_-$), then the generalized Weyl module $W_{\lambda}$ is isomorphic to the (local) Weyl module $W(w_0\lambda)$ (see \cite{FeMa3}).
\end{rem}

\subsection{Definitions}
For an anti-dominant weight $\lambda_-$ and an element $\sigma\in W$ we define two modules
$U_{\sigma(\la_-)}$ and ${\mathbb U}_{\sigma(\la_-)}$ as follows:
$U_{\sigma(\la_-)}$ is the cyclic $\mathcal{I}$-module with cyclic vector $u_{\sigma(\la_-)}$ of $\fh$-weight
$\sigma(\la_-)$
subject to the relations:
\begin{gather*}
(h\otimes 1)  u_{\sigma(\la_-)}=\langle h, \sigma(\la_-) \rangle u_{\sigma(\la_-)},\\
\fh\T z\bC[z] u_{\sigma(\la_-)}=0,\\
(e_{\widehat{\sigma}(-\alpha+\delta)+r \delta}) u_{\sigma(\la_-)}=0,\ \al\in\Delta_+,r \geq 0, \\
 (f_{\sigma(\al)}\T z)^{-\bra \la_-,\al^\vee\ket+1} u_{\sigma(\la_-)}=0,\ \al\in\Delta_+,
 \sigma\al\in\Delta_-,\\
(e_{\sigma(\al)}\T 1)^{-\bra \la_-,\al^\vee\ket} u_{\sigma(\la_-)}=0,\ \al\in\Delta_+, \sigma\al\in\Delta_+.
\end{gather*}
The definition of the $\mathcal{I}$-module ${\mathbb U}_{\sigma(\lambda_-)}$ differs from the definition of the
$U_{\sigma(\la_-)}$ by removing the first line relation: $\fh\T z\bC[z] u_{\sigma(\la_-)}=0$.

In the similar way we define modules $D_{\sigma(\la_-)}$ and ${\mathbb D}_{\sigma(\la_-)}$ as follows:
$D_{\sigma(\la_-)}$ is the cyclic $\mathcal{I}$-module with cyclic vector $d_{\sigma(\la_-)}$ of $\fh$-weight
$\sigma(\la_-)$
subject to the relations:

\begin{gather*}
(h\otimes 1)  d_{\sigma(\la_-)}=\langle h, \sigma(\la_-) \rangle d_{\sigma(\la_-)},\\
\fh\T z\bC[z] d_{\sigma(\la_-)}=0,\\
(e_{\widehat{\sigma}(-\alpha+\delta)+r \delta}) d_{\sigma(\la_-)}=0,\ \al\in\Delta_+,r \geq 0, \\
 (f_{\sigma(\al)}\T z)^{-\bra \la_-,\al^\vee\ket} d_{\sigma(\la_-)}=0,\ \al\in\Delta_+,
 \sigma\al\in\Delta_-,\\
(e_{\sigma(\al)}\T 1)^{-\bra \la_-,\al^\vee\ket+1} d_{\sigma(\la_-)}=0,\ \al\in\Delta_+, \sigma\al\in\Delta_+.
\end{gather*}
The definition of the $\mathcal{I}$-module ${\mathbb D}_{\sigma(\lambda_-)}$ differs from the definition of the
$D_{\sigma(\la_-)}$ by removing the first line relation: $\fh\T z\bC[z] d_{\sigma(\la_-)}=0$.

\begin{rem}
The $\mathcal{I}$-modules  ${\mathbb U}_{\sigma(\lambda_-)}$, 
${U}_{\sigma(\lambda_-)}$, ${\mathbb D}_{\sigma(\lambda_-)}$ and ${D}_{\sigma(\lambda_-)}$ depend only on the weight 
$\sigma(\la_-)$, but not on $\sigma$ and $\lambda_-$ separately. Hence it makes sense to talk about the modules 
${\mathbb U}_{\la}$,  
${U}_\lambda$, ${\mathbb D}_{\lambda}$ and ${D}_{\lambda}$ for arbitrary $\la\in P$.
\end{rem}

\begin{lem}\label{CharacterD}
In types $ADE$ the modules $D_\lambda$ are isomorphic to the level one Demazure modules. In particular, 
\begin{equation}\label{Demch}
\ch(D_\lambda)=E_{\lambda}(X;q,0)
\end{equation}
where $E_\lambda(X;q,t)$ are the nonsymmetric Macdonald polynomials corresponding to the (untwisted) affinization of $\fg$ \cite{Ch1,Ch2}.
\end{lem}
\begin{proof}
The defining relations for Demazure modules are found in \cite{J,Ma}. For the Demazure modules for the affine Kac-Moody Lie algebras the relations were rewritten in \cite{FL}. In the case of the level one case the relations from \cite{FL} read as
\begin{gather*}
(e_{\sigma(\al)}\T z^s)^{k+1}d_{\sigma(\la_-)}=0,\ \sigma(\al)\in\Delta_+, s\ge 0, k=\max(0, -\bra \la_-,\al^\vee\ket -s),\\
(f_{\sigma(\al)}\T z^s)^{k+1}d_{\sigma(\la_-)}=0,\ \sigma(\al)\in\Delta_-, s > 0, k=\max(0, -\bra \la_-,\al^\vee\ket -s),\\
\fh\T z\bC[z] d_{\sigma(\la_-)}=0,\quad (h\T 1) d_{\sigma(\la_-)}=\sigma(\la_-)(h) d_{\sigma(\la_-)}\ \forall h\in \fh. 
\end{gather*}
One needs to show that the defining relations of the modules $D_{\sigma(\la_-)}$ are contained in the relations above and  that the relations of $D_{\sigma(\la_-)}$ generate the ideal of relations of the Demazure module. To prove the first claim it suffices to note that  $\max(0, -\bra \la_-,\al^\vee\ket -s)=0$ unless $\al^\vee\in\Delta_+$. The second claim follows from the standard computation in the 
$\msl_2$ case.

Finally, formula  \eqref{Demch} in types $ADE$ is proven in \cite{Sa,I}.
\end{proof}

\begin{rem}
We note that if $\fg$ is not of the ADE type, then $D_\lambda$ is not isomorphic to the corresponding Demazure module. However, we expect that the equality \eqref{Demch} holds true in all types.  
\end{rem}

For the $U$-modules we have the following analogue of \eqref{Demch}
proved in \cite{FKM} for all types:
\[\ch(U_\lambda)=w_0 E_{w_0(\lambda)}(X;q^{-1},\infty).\]

\subsection{Generalized Weyl modules with characteristic}\label{gWc}

Let $\la_-,\mu\in P_-$ be such that $\la_- - \mu\in P_-$. We fix a reduced
decomposition
\begin{equation}\label{pi}
t_\mu=\pi s_{j_1}\dots s_{j_l},\ \pi\in\Pi,\ j_1,\dotsc,j_l\in I\cup\{0\},\ l=\ell(t_\mu)
\end{equation}
in the extended affine Weyl group $W_{\mathrm{af}}=P\rtimes W$, where $\Pi\subset W_{\mathrm{af}}$ is the set of length zero elements (i.e. $\Pi\simeq P/Q$). 
We consider the affine coroots $\beta_1,\dots,\beta_l$ defined by
\begin{equation}\label{betas}
\beta_l=\al_{j_l}^\vee, \beta_{l-1}=s_{j_l}\al_{j_{l-1}}^\vee,\dots, \beta_1=s_{j_l}\dots
s_{j_2}\al_{j_1}^\vee
\end{equation}
(see \cite{RY,OS}). 
We have the decomposition $\beta_j=\bar\beta_j+ ({\rm deg}\, \beta_j )\delta^\vee$, where
$\bar \beta_j\in \Delta^\vee$, ${\rm deg} \, \beta_j \in \mathbb Z$, and $\delta^\vee$ is the imaginary coroot. We note that  $\bar\beta_j$ is always a negative coroot and ${\rm deg}\,
\beta_j>0$ since (\ref{pi}) is a reduced expression and $\mu\in P_-$.

For a positive root $\al$ and a number $m=1,\dots,l$ we define
\begin{equation}\label{l}
l_{\al,m}=-\bra\la_-,\al^\vee\ket - |\{j:\ \bar\beta_j=-\al^\vee, 1\le j\le m\}|.
\end{equation}

\begin{dfn}[\cite{FeMa3,FMO}]
The generalized Weyl module with characteristics $W_{\sigma(\la_-)}(m)$ is the ${\mathcal I}$ module which is the quotient of $W_{\sigma(\la_-)}$ by the submodule generated by
\begin{equation}\label{charrel}
e_{{\widehat \sigma}(\al)}^{l_{\al,m}+1} w_{\sigma(\la_-)}, \ \al\in\Delta_+
\end{equation}
(recall that $w_{\sigma(\la_-)}$ is the cyclic vector of $W_{\sigma(\la_-)}$). Similarly, we define the generalized global Weyl module with characteristics $\mathbb W_{\sigma(\la_-)}(m)$ as the ${\mathcal I}$ module being the quotient of $\mathbb W_{\sigma(\la_-)}$ by the submodule generated by (\ref{charrel}) inside $\mathbb W_{\sigma(\la_-)}$.
\end{dfn}

We need the following theorem from \cite{FKM}

\begin{thm} \label{characters}
Let $\la_-,\mu, \la_- - \mu \in P_-$ and let $\beta_j$ be defined by \eqref{betas}. Then 
\[
\ch \, \W_{\sigma(\la_-)}(m)=\frac{\ch \, W_{\sigma(\la_-)}(m)}{(q)_{\la_- + \omega(m)}},
\]
where
$$\omega ( m ) := \sum_{\substack{1 \le i \le m\\ -\bar\beta_i=\al_{j}^\vee \text{ is simple}}} \omega_{j}.$$
\end{thm}

Let us define the  algebra $\mathcal{A}_{\sigma(\lambda_-),m}$.
\begin{dfn}\label{HighestWeightAlgebraDefinition}
Let $\la_-\in P_-$, $\sigma\in W$ and $m\ge 0$. Then $\mathcal{A}_{\sigma(\lambda_-),m}$ is
the quotient of the universal enveloping algebras $\U(z\fh[z])$ by the ideal annihilating the cyclic vector $w_{\sigma(\la_-)}\in W_{\sigma(\la_-)}(m)$. 
\end{dfn}

\begin{rem}
Definition \ref{HighestWeightAlgebraDefinition} implies that the weight $\sigma(\la_-)$
subspace of the generalized Weyl module with characteristics $W_{\sigma(\la_-)}(m)$ is
isomorphic to the algebra $\mathcal{A}_{\sigma(\lambda_-),m}$ as a vector space.
\end{rem}

\begin{lem}\label{HighestWeightAlgebraAction}
We have an isomorphism of $z$-graded algebras
\begin{equation}\label{Ah}
\mathcal{A}_{\sigma(\lambda_-),m}:=\Bbbk[h_{\sigma(\alpha_j)}\otimes z^k]|_{j\in I,~k=1,\dots,-\langle \lambda_-,\alpha_j^\vee \rangle-|\{i=1,\dots,m:\ -\bar \beta_i=\alpha_j^\vee\}|}.
\end{equation}
\end{lem}
\begin{proof}
Clearly $\fh$ is a spanned by the vectors of the form $h_{\sigma(\alpha_j)}$ for
$j=1,\dots,n$.
Therefore 
\[\U(z\mathfrak{h}[z])=\Bbbk[h_{\sigma(\alpha_j)} \otimes z^k],j\in I, k=1,\dots.\]

By standard arguments based on the PBW-theorem  the $\sigma(\lambda_-)$-weight subspace of $\W_{\sigma(\la_-)}(m)$ is equal to 
$\U(z\mathfrak{h}[z])w_{\sigma(\la_-)}$.
By $\msl_2$ computations for any positive $l$ the vector $h_{\sigma(\alpha_j)} \otimes z^l. w_{\sigma(\la_-)}$
belongs to the space 
\[
\Bbbk[h_{\sigma(\alpha_j)} \otimes z^k,\  k=1,\dots, -\langle \lambda_-,\alpha_j^\vee \rangle-|\{i=1,\dots,m:\ -\bar \beta_i=\alpha_j^\vee\}|]. w_{\sigma(\la_-)},\]
Therefore the $\sigma(\lambda_-)$-weight subspace of $\W_{\sigma(\la_-)}(m)$ is obtained by applying the right hand side of \eqref{Ah} to the cyclic vector $w_{\sigma(\la_-)}$.
However by Theorem \ref{characters} the $z$-characters of the right hand side of \eqref{Ah} and of the  $\sigma(\lambda_-)$-weight subspace of $\W_{\sigma(\la_-)}(m)$ are equal. 
This implies \eqref{Ah}. 
\end{proof}

For any $f\in \U(\mathcal{I})$ the map $w_{\sigma(\la_-)}\mapsto fw_{\sigma(\la_-)}$ defines an endomorphism 
\[\U(\mathcal{I})w_{\sigma(\la_-)}\rightarrow \U(\mathcal{I}) fw_{\sigma(\la_-)}\]
of module $\W_{\sigma(\la_-)}(m)$. Therefore the algebra $\mathcal{A}_{\sigma(\lambda_-),m}$ acts by endomorphisms on the module $\W_{\sigma(\la_-)}(m)$ and we can consider $\W_{\sigma(\la_-)}(m)$ as $(\mathcal{I},\mathcal{A}_{\sigma(\lambda_-),m})$-bimodule.

\begin{thm}
The right action of $\mathcal{A}_{\sigma(\lambda_-),m}$ is free on the $(\U(\mathcal{I}),\mathcal{A}_{\sigma(\lambda_-),m})$-bimodule $\W_{\sigma(\la_-)}(m)$.
\end{thm}
\begin{proof}
By PBW theorem and Lemma \ref{HighestWeightAlgebraAction} we have:
\[\W_{\sigma(\la_-)}(m)=\U(\widehat{\sigma}(z\mathfrak{n}_-[z]))w_{\sigma(\la_-)}\mathcal{A}_{\sigma(\lambda_-),m}.\]
By definition we have the following isomorphism of the left $\mathcal{I}$-modules
\[\W_{\sigma(\la_-)}(m)\otimes_{\mathcal{A}_{\sigma(\lambda_-),m}}\Bbbk\simeq W_{\sigma(\la_-)}(m).\]

Therefore the right $\mathcal{A}_{\sigma(\lambda_-),m}$-module $\W_{\sigma(\la_-)}(m)$ has the generating space with the character equal to $\ch(W_{\sigma(\la_-)}(m))$. Thus using Theorem \ref{characters} we get that this right module is free.
\end{proof}

\subsection{Identification with generalized Weyl modules with characteristics}\label{NM}
We fix $\la_- \in P_-$, and we assume that $\sigma \in W$ is the maximal length element in the class $\sigma\cdot {\rm stab}_W
(\la_-)$. We set $\la'=w_0\sigma(\la_-)$. Then, $v (\la')=\sigma^{-1}w_0$ is the shortest element such that
$v ( \la' ) \la'=\la_-$. In addition, 
the factorizations $t _{\la_-} = v ( \la' ) u(\la')$
and $t _{\la_-} = \sigma u'(\la')$ refine to reduced decompositions (see e.g. \cite{FKM}).

If we fix reduced expressions
\[
v ( \la' )=s_{i_1}\dots s_{i_r},\ u(\la')=\pi s_{i_{r+1}}\dots s_{i_M},
\]
then we obtain a reduced expression
\begin{equation}\label{reddecU}
t_{\la_-}=\pi s_{\pi^{-1}i_1}\dots s_{\pi^{-1} i_r} s_{i_{r+1}}\dots s_{i_M}.
\end{equation}

In the similar way we fix reduced expressions
\[
\sigma=s_{i_1'}\dots s_{i_{r'}'},\ u'(\la')=\pi' s_{i'_{r'+1}}\dots s_{i'_M},
\]
and we obtain a reduced expression
\begin{equation}\label{reddecV}
t_{\la_-}=\pi' s_{\pi'^{-1}i'_1}\dots s_{\pi'^{-1} i'_{r'}} s_{i'_{r+1}}\dots s_{i'_M}.
\end{equation}


\begin{prop}\label{UWchar}
Let the coroots $\beta_j$ be defined by the reduced decomposition \eqref{reddecU}. Then
$$U_{\sigma(\la_-)} \simeq W_{\sigma(\la_-)}(\ell(w_0)-\ell(\sigma)),$$
$$\mathbb{U}_{\sigma(\la_-)} \simeq \W_{\sigma(\la_-)}(\ell(w_0)-\ell(\sigma)).$$

Let the coroots $\beta_j$ be defined by the reduced decomposition \eqref{reddecV}. Then
$$D_{\sigma(\la_-)} \simeq W_{\sigma(\la_-)}(\ell(\sigma)),$$
$$\mathbb{D}_{\sigma(\la_-)} \simeq \W_{\sigma(\la_-)}(\ell(\sigma)).$$
\end{prop}

\begin{rem}
The modules $W_{\sigma(\la_-)}$ depend only on $\sigma(\la_-)$, but not on $\sigma$ and $\la_-$
separately. We note that the choice of $\sigma$ being maximal length element in the class $\sigma\cdot {\rm stab}_W (\la_-)$ gives the parametrization of the Weyl modules by such pairs $(\sigma, \la_-)$.
\end{rem}

\begin{proof}[Proof of Proposition \ref{UWchar}]
We have to prove that the defining relations \eqref{charrel} of $W_{\sigma(\la_-)}(\ell(w_0)-\ell(\sigma))$ coincide
with the defining relations of $U_{\sigma(\la_-)}$. We set $r:=\ell(v(\la'))=\ell(\sigma^{-1}w_0)=\ell(w_0)-\ell(\sigma)$,
which is the cardinality of the set $\Delta_+\cap\sigma^{-1} \Delta_+$.

It suffices to show that $\{-\bar\beta_1^\vee,\dots,-\bar\beta_r^\vee\}=\Delta_+\cap\sigma^{-1} \Delta_+$. By definition, for $k=1,\dots,r$ we
have
\begin{eqnarray*}
\beta_k^\vee & = & s_{i_M}\dots s_{i_{r+1}} s_{\pi^{-1}i_r}\dots s_{\pi^{-1}i_{k+1}} \al_{\pi^{-1}i_k}^\vee\\
&= & t_{-\la_-} \pi s_{\pi^{-1}i_1}\dots s_{\pi^{-1}i_{k-1}} s_{\pi^{-1}i_k}  \al_{\pi^{-1}i_k}^\vee \\
&= & t_{-\la_-} s_{i_1}\dots s_{i_{k-1}} (-\al_{i_k}^\vee).
\end{eqnarray*}
The action of $t_{-\la_-}$ on $\Delta^{a}$ preserve the finite (bar) part of an affine coroot. Therefore, the negated finite parts of $\beta_1^\vee,\dots,\beta_r^\vee$ are exactly the positive roots which are mapped to negative roots by $w_0\sigma$. Therefore, the comparison of the defining equations yield
$$U_{\sigma(\la_-)} \simeq W_{\sigma(\la_-)}(r),$$
$$\mathbb{U}_{\sigma(\la_-)} \simeq \W_{\sigma(\la_-)}(r),$$
that is the first part of the assertion.

The proof of the second assertion is completely similar.
\end{proof}

We define $\mathcal{A}^U_{\sigma({\lambda_-})}:=\mathcal{A}_{\sigma({\lambda_-}),\ell(w_0)-\ell(\sigma)}$ with respect to the reduced decomposition \eqref{reddecU},  $\mathcal{A}^D_{\sigma({\lambda_-})}:=\mathcal{A}_{\sigma({\lambda_-}),\ell(\sigma)}$
 with respect to the reduced decomposition \eqref{reddecV}. Note that we have:
\begin{equation}\label{AU}
\mathcal{A}^U_{\sigma(\lambda_-)}=\Bbbk[h_{\sigma(\alpha_j)}\otimes z^k]_{j=1,\dots,n,~k=1,\dots,-\langle \lambda_-,\alpha_j^\vee \rangle-\delta_{j,\sigma}},
\end{equation}
where $\delta_{j,\sigma}=1$, if $\sigma(\alpha_j)\in \Delta_+$ and $0$ otherwise.
Also we have:
\begin{equation}\label{AD}
\mathcal{A}^D_{\sigma(\lambda_-)}=\Bbbk[h_{\sigma(\alpha_j)}\otimes z^k]_{j=1,\dots,n,~k=1,\dots,-\langle \lambda_-,\alpha_j^\vee \rangle-1+\delta_{j,\sigma}}.
\end{equation}

\begin{rem}
One easily sees that the algebras $\mathcal{A}^U_{\sigma(\lambda_-)}$ and $\mathcal{A}^D_{\sigma(\lambda_-)}$ depend on the weight $\lambda=\sigma(\lambda_-)\in P$, but not on $\sigma$ and $\lambda_-$ separately. 
\end{rem}

\begin{cor}\label{lambdasigma}
The algebra $\mathcal{A}^U_{\sigma({\lambda_-})}$ acts freely on ${\mathbb U}_{\sigma(\la_-)}$ and
\[
\ch \, {\mathbb U}_{\sigma(\la_-)}= \ch \, U_{\sigma(\la_-)}\ch\mathcal{A}^U_{\sigma({\lambda_-})}.
\]
The algebra $\mathcal{A}^D_{\sigma({\lambda_-})}$ acts freely on ${\mathbb D}_{\sigma(\la_-)}$ and
\[
\ch \, {\mathbb D}_{\sigma(\la_-)}= \ch \, D_{\sigma(\la_-)}\ch \mathcal{A}^D_{\sigma({\lambda_-})}.
\]
\end{cor}

\subsection{Right modules}
In this subsection we define right modules $U_{\la}^o$ and ${\mathbb U}_{\la}^o$.

For an anti-dominant weight $\lambda_-$ and an element $\sigma\in W$ the right modules 
$U_{\sigma(\la_-)}^o$ and ${\mathbb U}_{\sigma(\la_-)}^o$ are defined as follows:
$U_{\sigma(\la_-)}^o$ is the cyclic right $\mathcal{I}$-module with cyclic vector $u_{\sigma(\la_-)}^o$ of $\fh$-weight
$\sigma(\la_-)$
subject to the relations:
\begin{gather*}
u_{\sigma(\la_-)}^o\fh\T z\Bbbk[z] =0,\\
u_{\sigma(\la_-)}^o(e_{\widehat{\sigma}(\alpha)+r \delta}) =0,\ \al\in\Delta_+,r \geq 0 \\
 u_{\sigma(\la_-)}^o (f_{\sigma(\al)}\T z)^{\bra \la_-,\al^\vee\ket+1}=0,\ \al\in\Delta_-,
 \sigma\al\in\Delta_-,\\
u_{\sigma(\la_-)}^o(e_{\sigma(\al)}\T 1)^{\bra \la_-,\al^\vee\ket} =0,\ \al\in\Delta_-, \sigma\al\in\Delta_+.
\end{gather*}
As before the definition of the $\mathcal{I}$-module ${\mathbb U}^o_{\sigma(\lambda_-)}$ differs from the definition of the
$U^o_{\sigma(\la_-)}$ by removing the first line relation.


\begin{prop}\label{UUo}
We have an isomorphism of $z$-graded vector spaces $U_{\sigma(\la_-)}^o\simeq U_{-\sigma(\la_-)}$. The right $\mathcal{I}$-module structure on the space  $U_{-\sigma(\la_-)}$ is induced by the minus identity map on the Iwahori algebra; this right module is isomorphic to $U_{\sigma(\la_-)}^o$.
\end{prop}
\begin{proof}
The minus identity map induces a structure of the right $\mathcal{I}$-module on the space  $U_{-\sigma(\la_-)}$. The $\fh$ weight of the cyclic vector of this right module is equal to $-\sigma(\la_-)=(\sigma w_0)\la_{-,*}$, where
$\la_{-,*}=-w_0\la_-\in P_-$. Now the defining relations of the right module $U_{\sigma(\la_-)}^o$ are obtained from the defining relations of the left module 
$U_{-\sigma(\la_-)}$ taking into account the change of variables $\sigma\to \sigma w_0$, $\la_-\to \la_{-,*}$. 
\end{proof}

\begin{cor}
The module ${\mathbb U}_{\sigma(\la_-)}^o$ has a structure of   $(\mathcal{A}^U_{-\sigma(\lambda_-)},\mathcal{I})$-bimodule which is free as the left $\mathcal{A}^U_{-\sigma(\lambda_-)}$-module.
\end{cor}

\begin{cor}\label{CharacterRightU}
For any $\lambda \in P$ one has
\[\ch(U_{\lambda}^o)=E_{\lambda}(Y;q^{-1},\infty).\]
\end{cor}
\begin{proof}
Proposition \ref{UUo} implies 
\[\ch(U_{\sigma(\la_-)}^o)=\ch(U_{-\sigma(\la_-)}|_{X_i\mapsto Y_i^{-1}}).\]
Hence we have
\[\ch(U_{\sigma(\la_-)}^o) =w_0 E_{-w_0\sigma(\lambda)}(X;q^{-1},\infty)_{X_i\mapsto Y_i^{-1}}.\]
Now it suffices to note that the right hand side is equal to $E_{\lambda}(Y;q^{-1};\infty)$ (since the transformation  $-w_0$ is induced by an  automorphism of the Dynkin diagram). 
\end{proof}

\subsection{The $\mgl_n$ version}
We recall the definition of
the global Weyl module over $\gl_n$. Let $\lambda=\sum_{i=1}^n \la_i \varepsilon_i$ be a dominant integral weight for $\mgl_n$.
We define the global Weyl module as follows. Let $\bar \la=\sum_{i=1}^{n-1} (\la_i-\la_{i+1})\om_i$ be the corresponding
$\msl_n$ weight and consider the global Weyl module $\mathbb{W}_{\bar\lambda}$ for $\msl_n$. 

Recall the notation $h_\ell=E_{\ell,\ell}$, $\ell=1,\dots,n$; in particular, 
$\mgl_n[z]=\msl_n[z]\oplus \Bbbk h_\ell[z]$.
In order to give the definition of the global Weyl module for $\mgl_n$ we need one more piece of notation. We define
\begin{equation}\label{varphiDefinition}
\varphi_k^\ell:\Bbbk[h_\ell z^i]_{i\ge 0} \to \Bbbk[s_1,\dots,s_k],~ h_\ell z^i\mapsto s_1^i+\dots +s_k^i.
\end{equation}
Then the $\mgl_n[z]$ module $\mathbb{W}_\la$ is defined as
\begin{equation}\label{Wmgln}
\mathbb{W}_{\lambda}=\mathbb{W}_{\bar\lambda} \otimes \Bbbk[h_nz^i]_{i\ge 0}/\ker \varphi_{\la_n}^n,
\end{equation}
where $\msl_n[z]$ acts on the first tensor factor and $\Bbbk h_n[z]$ acts on the second.

Now let us define the modules $\mathbb{U}_{\lambda}$ and $\mathbb{D}_{\lambda}$. Let $\ell'=1,\dots,n$ be a number such that:
\[\lambda_{\ell'}=\lambda_{\min}:=\min\{\lambda_\ell\}.\]
Then 
we define
\begin{equation}\label{Umgln}
\mathbb{U}_{\lambda}=\mathbb{U}_{\bar\lambda} \otimes \Bbbk[h_{\ell'}z^i]_{i\ge 0}/\ker \varphi_{\la_{\min}}^{\ell'},
\end{equation}
\begin{equation}\label{Vmgln}
\mathbb{D}_{\lambda}=\mathbb{D}_{\bar\lambda} \otimes \Bbbk[h_{\ell'}z^i]_{i\ge 0}/\ker \varphi_{\la_{\min}}^{\ell'}.
\end{equation}
In a similar fashion one defines the right modules $\mathbb{U}^o_{\lambda}$.

The definitions of the local modules for $\mgl_n$ do not differ much from their $\msl_n$ analogues. More precisely, the defining relations for the  local modules in type $\mgl_n$ 
are obtained from that in type $\msl_n$ by adding the conditions that $z\Bbbk h_\ell[z]$ kills the cyclic vector and the $\mgl_n$ Cartan subalgebra acts via the weight $\lambda$. Clearly in type $\mgl_n$ the modules $\mathbb{D}_{\lambda}$, $\mathbb{U}_{\lambda}$ are free over its highest weight algebras. Moreover we have
\begin{equation}\label{CharactersACompare}
\ch \mathcal{A}_\lambda^D=\ch \mathcal{A}_{\bar\lambda}^D(q;q)_{\lambda_{\min}}^{-1},~\ch \mathcal{A}_\lambda^U=\ch \mathcal{A}_{\bar\lambda}^U(q;q)_{\lambda_{\min}}^{-1}.
\end{equation}

Recall that the characters of left weight modules in type $\mgl_n$ belong to $\Bbbk[x_1, \dots, x_n, q]$, the characters of the right ones belong to $\Bbbk[y_1, \dots, y_n, q]$.

\begin{prop}
For a $\mgl_n$-weight  $\lambda$ we have
\[\ch(D_{\lambda})=E_{\lambda}(x_1,\dots, x_n;q;0),\]
\[\ch(U_{\lambda}^o)=E_{\lambda}(y_1,\dots, y_n;q^{-1};\infty).\]
\end{prop}
\begin{proof}
As graded vector spaces the module $D_{\lambda}$ (resp., $U_{\lambda}^o$) is isomorphic to  $D_{\bar\lambda}$ (resp., $U_{\bar\lambda}^o$). Now the desired character formulas are implied by Corollary \ref{CharacterRightU}, Lemma \ref{CharacterD}  and the following equalities:
\begin{gather*}
\ch(D_{\lambda+(1,\dots,1)})=x_1\dots x_n \ch(D_{\lambda}),\
\ch(U^o_{\lambda+(1,\dots,1)})=y_1\dots y_n \ch(D_{\lambda}),\\
E_{\lambda+(1,\dots,1)}(x_1,\dots, x_n;q;t) = x_1\dots x_nE_{\lambda}(x_1,\dots, x_n;q;t).
\end{gather*}
\end{proof}

\subsection{Tensor bimodules}

\begin{prop}
There is a natural isomorphism
$\mathcal{A}^U_{\sigma(\lambda_-)}\simeq \mathcal{A}^D_{-\sigma(\lambda_-)}$
\end{prop}
\begin{proof}
Recall \eqref{AU}, \eqref{AD}
\begin{gather*}    \mathcal{A}^U_{\sigma(\lambda_-)}:=\Bbbk[h_{\sigma(\alpha_j)}\otimes z^k]_{j\in I,~k=1,\dots,\langle \lambda_-, \alpha_j^\vee \rangle-\delta_{j,\sigma}},\\
\mathcal{A}^D_{\sigma(\lambda_-)}=\Bbbk[h_{\sigma(\alpha_j)}\otimes z^k]_{j\in I,~k=1,\dots,-\langle \lambda_-,\alpha_j^\vee \rangle-1+\delta_{j,\sigma}}
\end{gather*}
and the notation $\lambda_-^*=-w_0(\lambda_-)$, $\alpha_{j^*}=-w_0(\alpha_j)$.
    Note that $h_{\sigma(\alpha_j)}=-h_{-\sigma(\alpha_j)}$. Clearly $-\sigma(\lambda_-)=\sigma w_0(-w_0(\lambda_-))=\sigma w_0(\lambda_-^*)$ and $\delta_{j,\sigma}=1-\delta_{j^*,\sigma w_0}$. Thus the map $h_{\sigma(\alpha_j)}\mapsto -h_{\sigma w_0(\alpha_{j^*})}$ gives the isomorphism 
    \[\mathcal{A}^U_{\sigma(\lambda_-)} \rightarrow \mathcal{A}^D_{-\sigma(\lambda_-)}.\]
\end{proof}

The following $(\mathcal{I},\mathcal{I})$-bimodule plays a key role in this paper. We use the isomorphism from the previous proposition.
\begin{dfn}\label{Tmpdoles}
    For a weight $\lambda \in P$ we define the bimodule $\mathbb{T}_\lambda$ in the following way:
    \[\mathbb{T}_\lambda:=\mathbb{D}_\lambda \otimes_{\mathcal{A}^D_{\lambda} }\mathbb{U}_{\lambda}^o. \]
\end{dfn}

Recall that for graded weight $(\fh,\fh)$-bimodules the characters belong to $\ZZ[P\times P][[q]]$.


\begin{lem}\label{TCharacter}
    The character of $\mathbb{T}_\lambda$ is equal to
    \[\ch(D_{\lambda})\ch(\mathcal{A}_{\lambda}^D)\ch(U_{\lambda}^o)=\ch(\mathcal{A}_{\lambda}^D)E_{\lambda}(X;q,0)E_{\lambda}(Y;q^{-1},\infty).\]
\end{lem}
\begin{proof}
Let $d_1, \dots, d_l$ be a weight basis of $\mathbb{D}_{\lambda}$, $u_1, \dots, u_{l'}$ be a weight basis of $\mathbb{U}_{\lambda}^o$ as a free modules over $\ch(\mathcal{A}_{\lambda}^D)$. We have the following isomorphisms of graded vector spaces:
\[\mathrm{span}\{d_1, \dots, d_l\}\simeq D_{\lambda},\ \mathrm{span}\{u_1, \dots, u_{l'}\}\simeq U_{\lambda}^o.\]
Then 
\[\mathbb{D}_\lambda=\bigoplus_{i=1}^l d_i\mathcal{A}_{\lambda}^D,\ \mathbb{U}_\lambda^o=\bigoplus_{j=1}^{l'} \mathcal{A}_{\lambda}^Du_j\]
and
\[\mathbb{T}_\lambda=\bigoplus_{i=1, \dots, l, j=1, \dots, l'} d_i\mathcal{A}_{\lambda}^Du_j.\]
This completes the proof.
\end{proof}

\begin{lem} \label{Trel}
    The bimodule $\mathbb{T}_\lambda$ is the cyclic bimodule with the generator $\bar w_\lambda:=d_\lambda \otimes u_{\lambda}$ and the following defining relations:
\begin{gather}\label{BimoduleCyclicRelationsFirst}
\widehat\sigma(\fn_-[z])\bar w_\la=0=  \bar w_\la\widehat\sigma(\fn_+[z]),\\
hz^0\bar w_\la=-\bar w_\la hz^0=\la(h)\bar w_\la,\ h\in\fh,\\
e_{\widehat \sigma (\alpha)}^{\langle \la, \al^{\vee} \rangle+1}\bar w_{\lambda}=0,\
\bar w_{\lambda}e_{\widehat \sigma (-\alpha)}^{\langle \la, \al^{\vee} \rangle+1}=0, ~\text{if}~ \alpha, \sigma(\alpha) \in \Delta_+,\\
e_{\widehat \sigma (\alpha)}^{\langle \la, \al^{\vee} \rangle}\bar w_{\lambda}=0,\
\bar w_{\lambda}e_{\widehat \sigma (-\alpha)}^{\langle \la, \al^{\vee} \rangle}=0, ~\text{if}~ \alpha, -\sigma(\alpha) \in \Delta_+,\\
h_{\sigma(\alpha_j)}\otimes z^k \bar w_{\lambda} = -\bar w_\la h_{\sigma(\alpha_{j})}\otimes z^k, ~ k\ge 0.\label{BimoduleCyclicRelationsLast}
\end{gather}

In the $\mgl_n$ case one more adds the relation
\begin{equation}\label{TGLRelation}
\ker\varphi^{\ell}_{\lambda_{\min}}\bar w_\lambda=\{0\},\ \la_{\min}=\la_{\ell}.
\end{equation}
\end{lem}

\begin{proof}
To prove the cyclicity we note that the $\U(\mathcal{I})$ span of the vector $d_\la\T u_\la$ with respect to the right action coincides with  
$d_\la\T {\mathbb U}_\la^o$. Now applying $\U(\mathcal{I})$ on the left one gets the whole bimodule $\mathbb{T}_\lambda$.

We denote the module defined by relations \eqref{BimoduleCyclicRelationsFirst}-\eqref{BimoduleCyclicRelationsLast} by $\widetilde {\mathbb{T}}_\lambda$.
Clearly the relations 
 \eqref{BimoduleCyclicRelationsFirst}-\eqref{BimoduleCyclicRelationsLast} are satisfied on the bimodule $\mathbb{T}_\lambda$. Therefore we have the surjection  $\widetilde {\mathbb{T}}_\lambda \twoheadrightarrow \mathbb{T}_\lambda$.

By definition we have the morphism of the right modules:
\[\eta: \mathbb{U}_{\lambda}^o\rightarrow \bar w_{\lambda} \U(\mathcal{I})\subset \widetilde {\mathbb{T}}_\lambda, ~ u_{\lambda}\mapsto \bar w_{\lambda}.\]
Then the elements  $\eta( u_j)$ form a basis of $\mathbb{T}_{\lambda}$ as the left  $\mathcal{I}$-module. For each left submodule 
$\U(\mathcal{I})\eta(u_j)$ we have a morphism of left modules
\[\mathbb{D}_{\lambda}\rightarrow \U(\mathcal{I})\eta( u_j), d_{\lambda} \mapsto \eta( u_j).\]
Thus we have a surjection of left modules (graded by the left and right weights)
\[\mathbb{D}_{\lambda} \otimes_{\Bbbk} \mathrm{span}\{u_1, \dots, u_{l'}\} \twoheadrightarrow \widetilde{\mathbb{T}}_\lambda.\]
However the character of the left module $\mathbb{D}_{\lambda} \otimes_{\Bbbk} \mathrm{span}(u_1, \dots, u_{l'}) $ coincides with the character of $\mathbb{T}_\lambda$. This completes the proof.
\end{proof}

\section{Filtrations for General Linear group}\label{glfiltr}
\subsection{The space $M_n[z]^+$.}
Let $V=\mathrm{span}\{v_1,\dots, v_n\}$ be the standard $n$-dimensional representation of $\mgl_n$ with its weight basis, i. e. for the matrix units $E_{ij}$:
\[E_{ij}v_k=\delta_{j,k}v_i.\] 

Recall
\[\fb_+=\mathrm{span}\{E_{ij},\ i \leq j\}, \ \fn_-=\mathrm{span}\{E_{ij},\ i > j\}.\]

Let $V^o$ be the standard right representation with the basis $v_1^o,\dots, v_n^o$ such that 
\[v_k^o E_{ij}=\delta_{i,k}v_j^o.\]
Then the space $V \otimes_{\Bbbk}V^o$ has the natural structure of $(\gl_n,\gl_n)$-bimodule.
The space $V \otimes_{\Bbbk}V^o \otimes \Bbbk[z]$ has the structure of $(\gl_n[z],\gl_n[z])$-bimodule. More precisely consider the set of variables $v_{ij}^{(l)}$, $i,j=1,\dots,n, l =0,1,\dots$. Then we have the following action:
\[E_{ij}z^{l}v_{kk'}^{(l')}=\delta_{j,k}v_{ik'}^{(l+l')},~v_{kk'}^{(l')}E_{ij}z^{l}=\delta_{i,k'}v_{kj}^{(l+l')}.\]
Let
\[M_n:= V \otimes_{\Bbbk}V^o,\]
\[M_n[z]:= V \otimes_{\Bbbk}V^o\otimes \Bbbk[z].\]
Using the Leibnitz rule we have the bimodule structure on the symmetric power
$S^N\left(M_n[z] \right).$

For a composition $\la=(\la_1,\dots,\la_n)$we attach a $\mgl_n$ weight $\sum \la_i\varepsilon_i$
(recall that the simple roots for $\msl_n$ are given by $\alpha_i=\varepsilon_i-\varepsilon_{i+1}$).
For two compositions $\la,\mu$ of size $N$ we write $\la\ge\mu$ if $\la-\mu$ is a sum of simple roots with nonnegative integer coefficients.

Recall that  $x_i$ denote the formal exponential of the left weight $\varepsilon_i$ and $y_i$ denote the formal exponential of the right weight $\varepsilon_i$.
In particular
\[\ch(M_n[z])= \frac{(x_1+\dots+x_n)(y_1+\dots+y_n)}{1-q}.\]

The following theorem is proven in \cite{FKhM} (Theorem B).
\begin{thm}\label{HoweCurrent}
The bimodule $S^N\left( M_n[z] \right)$ admits a decreasing filtration $F_{\lambda}$
indexed by partitions $\la$ of length $n$ with $\lambda \vdash N$ such
that
$$F_{\lambda}/F_{>\lambda}\simeq \W_\la\T_{\mathcal{A}_\la} \W^o_{\la}.$$
\end{thm}
The filtration $F_\lambda$ is defined in the following way. 
The $(\gl_n,\gl_n)$-subbimodule  $S^N\left( M_n \right)\subset S^N\left( M_n[z] \right)$ has the decomposition:
\[S^N\left(M_n\right)=\bigoplus_{\lambda \in P_+, \lambda \vdash N}V_\lambda \otimes_\Bbbk V_\lambda^o.\]
Then $F_\lambda=\U(\gl_n[z])\left(V_\lambda \otimes_\Bbbk V_\lambda^o\right) \U(\gl_n[z])$. 

We denote by $\mathcal{I}$ the $\gl_n$ Iwahori subalgebra $\fb_+[z]\oplus z\fn_-[z]$. We consider the  Iwahori subbimodule
$M_n[z]^+\subset M_n[z]$ defined by
\[M_n[z] ^+:=\mathrm{span}\{ v_{ij}^{(l)},\ l\ge 0 \text{ for } i \leq j \text{ and }  l\ge 1 \text{ for } i> j\}.\]
The bimodule structure extends to the symmetric power $S^N(M_n[z] ^+)$ by the Leibnitz rule.

The next lemma is completely analogous to \cite{FKhM}(Lemma 3.4). We prove it in three steps, first two of which are the same as in \cite{FKhM} and the last one is very similar.
\begin{lem}
The $(\mathcal{I},\mathcal{I})$-bimodule $S^N\left( M_n[z]^+ \right)$ is generated by the  set 
\[S:=\left\{\prod_{i=1}^n(v_{ii}^{(0)})^{r_{ii}}, r_{ii}\ge 0, \sum_i r_{ii}=N\right\}.\]
\end{lem}
\begin{proof}

We divide the proof into three steps.

{\it Step 1.} 
The following equality holds true:
\[
\U(\fh[z])S=\mathrm{span} \left\{\prod_{i=1}^n\prod_{k=1}^{r_{ii}}v_{ii}^{(s_{ii}^{k})},\ s_{ii}^{k}\ge 0, \ \sum_i r_{ii}=N\right\}.\]
The proof of this fact is by induction on the number of nonzero exponents $s_{ii}^{k}$. Assume by induction that $\U(\fh[z])S$
contains all elements of the form $x=\prod_{i=1}^n\prod_{k=1}^{r_{ii}}(v_{ii}^{(s_{ii}^{k})})$ with the number of nonzero exponents $s_{jj}^{k}$
equal to $b<r_{jj}$. Then applying elements $h_{j}z^{a}$ to $x$ we obtain the sum of elements such that all of them but one have
the number of nonzero exponents $s_{jj}^{k}$ equal to $b$. The remaining summand is equal to the integer multiple of
\[
\prod_{\substack{1\le i\le n\\i \neq j}}\prod_{k=1}^{r_{ii}}v_{ii}^{(s_{ii}^{k})}\cdot
\prod_{k=1}^{r_{jj}-1}v_{jj}^{(s_{jj}^{k})}\cdot v_{jj}^{{(a)}}.
\]
This completes the induction step.

{\it Step 2.} Next we prove the following claim:
\begin{equation}
\label{Step2}
\U(\fn_+[z])\U(\fh[z])S=
\mathrm{span}\left\{\prod_{i\leq j}\prod_{k=1}^{r_{ij}}v_{ij}^{(s_{ij}^{k})} , s_{ij}^{k}\ge 0, \sum r_{ij}=N\right\}.
\end{equation}
For a pair $i_0\le j_0$ we denote by $S_{i_0, j_0}$ be the linear span of the set of monomials $\prod_{i\leq j}\prod_{k=1}^{r_{ij}}v_{ij}^{(s_{ij}^{k})}$
subject to the condition 
\[
\text{ if }  j>j_0 \text{ or } j=j_0, i<i_0, \text{ then } r_{ij}=0 \text{ provided } i\ne j.
\]
We prove \eqref{Step2} by showing by induction that  $S_{i_0,j_0}\subset \U(\fn_+[z])\U(\fh[z])S$ for all pairs $i_0,j_0$ (\eqref{Step2} is the $(i_0,j_0)=(1,n)$ case). The induction procedure works as follows. First, we note the obvious equality   
$S_{1,i_0}=S_{i_0+1,i_0+1}$. Keeping this in mind, we start with the trivial case $(i_0,j_0)=(1,1)$, then proceed with $(2,2)$, $(1,2)$, $(3,3)$, $(2,3)$, $(1,3)$ and so on (in other words, the induction step passes from  $S_{i_0+1, j_0}$ to $S_{i_0, j_0}$).



Consider the increasing filtration on the space $S_{i_0, j_0}$:
\[
G_p=\mathrm{span}\left\{\prod_{i\leq j}\prod_{k=1}^{r_{ij}}v_{ij}^{(s_{ij}^{k})}, r_{ij}=0 ~\text{for}~ j>j_0~ \text{or}~ j=j_0, i< i_0,
r_{i_0j_0}\leq p\right\}.
\]
We note that 
\begin{multline}\label{topterm}
\U(\Bbbk E_{i_0,i_0+1}[z]). \mathrm{span}\left\{\prod_{k=1}^{r_{i_0+1,j_0}}v_{i_0+1,j_0}^{(s_{i_0+1,j_0}^{k})},\ s_{i_0+1,j_0}^{k}\ge 0\right\} \\ 
=  \mathrm{span}\left\{\prod_{k=1}^{r_{i_0+1,j_0}}v_{i_0+1,j_0}^{(s_{i_0+1,j_0}^{k})} \prod_{k=1}^{r_{i_0,j_0}}v_{i_0,j_0}^{(s_{i_0,j_0}^{k})},\ s_{i_0+1,j_0}^{k}, s_{i_0,j_0}^{k} \ge 0\right\}
\end{multline}
(see e.g. \cite{FKhM},  Lemma 3.3). Now let $\U_{p'}\subset \U(\Bbbk E_{i_0,i_0+1}[z])$ be the standard (increasing) PBW filtration on the universal enveloping algebra (i.e. $\U_{p'}$ is spanned by monomials of length not exceeding $s$). Then $U_{p'} G_p\subset G_{p+p'}$ and
the top degree term (with respect to the number of elements of the form $v_{i_0,j_0}^{(s_{i_0,j_0})}$) comes from the action \eqref{topterm}. Hence  $G_p\subset \U(\fn_+[z])\U(\fh[z])S$ for all $p$.  


In a similar way one proves
\[\U(\fn_+[z])\U(\fh[z])S=S\U(\fh[z])\U(\fn_+[z]).\]

{\it Step 3.} We prove:
\begin{multline}\label{Step3}\U(z\fn_-[z])\U(\fn_+[z])\U(\fh[z])S\\=
\mathrm{span}\left\{\prod_{i,j}\prod_{k=1}^{r_{ij}}v_{ij}^{(s_{ij}^{k})} , s_{ij}^{k}\ge 0 \text{ for } i \leq j,~  s_{ij}^{k}\ge 1 \text{ for } i > j\right\}.
\end{multline}

For a pair $i_0\geq j_0$ we denote by $S'_{i_0, j_0}$ be the linear span of the set of monomials $\prod_{i, j}\prod_{k=1}^{r_{ij}}v_{ij}^{(s_{ij}^{k})}$
subject to the condition 
\[
\text{ if }  j<j_0 \text{ or } j=j_0, i>i_0, \text{ then } r_{ij}=0 \text{ provided } i\leq j.
\]
We prove \eqref{Step3} by showing by induction that  $S'_{i_0,j_0}\subset  \U(z\fn_-[z])\U(\fn_+[z])\U(\fh[z])S$ for all pairs $i_0,j_0$. The induction procedure works as follows. First, we note the obvious equality   
$S'_{i_0,n}=S'_{i_0-1, i_0-1}$. We start with the trivial case $(i_0,j_0)=(n,n)$, then proceed with $(n-1, n-1)$, $(n,n-1)$, and so on. 
As in the previous step we can prove $S'_{i_0,j_0}=\U(z\Bbbk E_{i_0,j_0}[z])S'_{i_0-1,j_0}$. This completes the proof by induction.
\end{proof}

\subsection{Checking of the relations}
Let  $\la$ be a composition of a positive integer $N$. We denote 
\begin{equation}
\mathcal{F}_\lambda:=\sum_{\mu \succeq \lambda}\U(\mathcal{I})\prod_{i=1}^n (v_{ii}^{(0)})^{\mu_i}\U(\mathcal{I})\subset S^N(M_n[z]^+), 
\end{equation}
where $\mu$ is a composition of $N$. Note that $\mathcal{F}_\la\supset \mathcal{F}_\nu$ 
provided $\nu\succeq\la$.

In the next several  Lemmas we check that relations
\eqref{BimoduleCyclicRelationsFirst}-\eqref{BimoduleCyclicRelationsLast} are satisfied on the associated graded module $\bigoplus_\la \mathcal{F}_\la/\sum_{\mu\succeq\la} \mathcal{F}_\mu$.
Recall the notation $\lambda=\sigma(\lambda_-)$.

\begin{lem}\label{extremalRelationsGl}
For $\alpha \in \Delta_-$:
\begin{gather*}
e_{\widehat \sigma (\alpha_i+\delta)}z^l \prod_{k=1}^n (v_{kk}^{(0)})^{\lambda_k}\in \sum_{\mu \succ \lambda} \mathcal{F}_\mu,\\
\prod_{k=1}^n (v_{kk}^{(0)})^{\lambda_k}e_{\widehat \sigma (-\alpha_i)}z^l\in \sum_{\mu \succ \lambda} \mathcal{F}_\mu.
\end{gather*}
\end{lem}
\begin{proof}
By definition we have $\varepsilon_i -\varepsilon_j \in \sigma(\Delta_-)$ if $\lambda_i \geq \lambda_j$. For $i,j$ such that $\lambda_i \geq \lambda_j$ and $l\geq 0$ if $i<j$, $l\geq 1$ if $i>j$ we have:
\[E_{ij}z^l \prod_{k=1}^n (v_{kk}^{(0)})^{\lambda_k}= \lambda_j (v_{jj}^{(0)})^{\lambda_j-1}v_{ij}^{(l)}
  \prod_{k=1, \dots, n, k \neq j} (v_{kk}^{(0)})^{\lambda_k}.\]
Analogously 
\[ \prod_{k=1}^n (v_{kk}^{(0)})^{\lambda_k}E_{ji}z^l=
\lambda_j  (v_{jj}^{(0)})^{\lambda_j-1}v_{ji}^{(l)}\prod_{k=1, \dots, n, k \neq j} (v_{kk}^{(0)})^{\lambda_k}.\]
However:
\begin{multline*}\lambda_j (v_{jj}^{(0)})^{\lambda_j-1}v_{ij}^{(l)} \prod_{k \neq j} (v_{kk}^{(0)})^{\lambda_k}=\\
\frac{\lambda_j}{\lambda_i+1} \left((v_{jj}^{(0)})^{\lambda_j-1}(v_{ii}^{(0)})^{\lambda_i+1}\prod_{k \neq i,j} (v_{kk}^{(0)})^{\lambda_k}\right)E_{ij}z^l
\end{multline*}
and the similar equality holds for the second case. By definition of the order $\succ$: $\lambda + \varepsilon_i-\varepsilon_j \succ \lambda$. This completes the proof.
\end{proof}

\begin{lem}\label{UDequationsGL}
For $\lambda_i<\lambda_j$, $i>j$:
\begin{equation}\label{LeftNegativeGl}
(E_{ij}z)^{\lambda_j-\lambda_i}\prod_{k=1}^n (v_{kk}^{(0)})^{\lambda_k}\in\sum_{\mu \succ \lambda} \mathcal{F}_\mu.
\end{equation}
\begin{equation}\label{RightNegativeGl}
\prod_{k=1}^n (v_{kk}^{(0)})^{\lambda_k}(E_{ji})^{\lambda_j-\lambda_i}\in\sum_{\mu \succ \lambda}\mathcal{F}_\mu.
\end{equation}
For $\lambda_i<\lambda_j$, $i<j$:
\begin{equation}\label{LeftPositiveGl}
(E_{ij})^{\lambda_j-\lambda_i+1}\prod_{k=1}^n (v_{kk}^{(0)})^{\lambda_k}\in\sum_{\mu \succ \lambda}\mathcal{F}_\mu.
\end{equation}
\begin{equation}\label{RightPositiveGl}
\prod_{k=1}^n (v_{kk}^{(0)})^{\lambda_k}(E_{ji}z)^{\lambda_j-\lambda_i+1}\in\sum_{\mu \succ \lambda}\mathcal{F}_\mu.
\end{equation}
\end{lem}
\begin{proof}
We have:
\begin{multline*}
(E_{ij})^{\lambda_j-\lambda_i+1}\prod_{k=1}^n (v_{kk}^{(0)})^{\lambda_k}\\=\binom{\lambda_j}{\lambda_i-1}(v_{jj}^{(0)})^{\lambda_i-1}(v_{ij}^{(0)})^{\lambda_j-\lambda_i+1}\prod_{k \neq j} (v_{kk}^{(0)})^{\lambda_k}\\
=\binom{\lambda_j}{\lambda_i-1}\left/\binom{\lambda_j+1}{\lambda_i}\right.\left((v_{jj}^{(0)})^{\lambda_i-1}(v_{ii}^{(0)})^{\lambda_j+1}\prod_{k \neq i,j}(v_{kk}^{(0)})^{\lambda_k}\right)
(E_{ij})^{\lambda_j-\lambda_i+1}.
\end{multline*}
This proves \eqref{LeftPositiveGl}.

In order to prove \eqref{RightPositiveGl} we compute
\begin{multline*}
\prod_{k=1}^n (v_{kk}^{(0)})^{\lambda_k}(E_{ji}z)^{\lambda_j-\lambda_i+1}\\=\binom{\lambda_j}{\lambda_i-1}(v_{jj}^{(0)})^{\lambda_i-1}(v_{ji}^{(1)})^{\lambda_j-\lambda_i+1}\prod_{k \neq j} (v_{kk}^{(0)})^{\lambda_k}\\
=\binom{\lambda_j}{\lambda_i-1}\left/\binom{\lambda_j+1}{\lambda_i}\right.(E_{ji}z)^{\lambda_j-\lambda_i+1}\left((v_{jj}^{(0)})^{\lambda_i-1}(v_{ii}^{(0)})^{\lambda_j+1}\prod_{k=1,\dots,n,~ k \neq i,j}(v_{kk}^{(0)})^{\lambda_k}\right).
\end{multline*}

Now we prove \eqref{LeftNegativeGl}. One has
\begin{multline*}
(E_{ij}z)^{\lambda_j-\lambda_i}\prod_{k=1}^n (v_{kk}^{(0)})^{\lambda_k}\\=
\binom{\lambda_j}{\lambda_i}(v_{jj}^{(0)})^{\lambda_i}(v_{ij}^{(1)})^{\lambda_j-\lambda_i}\prod_{k \neq j} (v_{kk}^{(0)})^{\lambda_k}\\
=\left((v_{jj}^{(0)})^{\lambda_i}(v_{ii}^{(0)})^{\lambda_j}\prod_{k \neq i,j}(v_{kk}^{(0)})^{\lambda_k}\right)
(E_{ij}z)^{\lambda_j-\lambda_i}.
\end{multline*}
Therefore $(E_{ij}z)^{\lambda_j-\lambda_i}\prod_{k=1}^n (v_{kk}^{(0)})^{\lambda_k}$ belongs to $\U(\mathcal{I})\prod_{k=1}^n (v_{kk}^{(0)})^{(s_{ij}\lambda)_k}\U(\mathcal{I})$.
However by assumption $s_{ij}\lambda\succ \lambda$. This completes the proof of \eqref{LeftNegativeGl}.

Finally, the following computations prove the equality \eqref{RightNegativeGl}:
\begin{multline*}
\prod_{k=1}^n (v_{kk}^{(0)})^{\lambda_k}(E_{ji})^{\lambda_j-\lambda_i}\\=
\binom{\lambda_j}{\lambda_i}(v_{jj}^{(0)})^{\lambda_i}(v_{ji}^{(0)})^{\lambda_j-\lambda_i}\prod_{k \neq j} (v_{kk}^{(0)})^{\lambda_k}\\
=(E_{ji})^{\lambda_j-\lambda_i}\left((v_{jj}^{(0)})^{\lambda_i}(v_{ii}^{(0)})^{\lambda_j}\prod_{k \neq i,j}(v_{kk}^{(0)})^{\lambda_k}
\right).
\end{multline*}
\end{proof}



The following Lemma is obvious.
\begin{lem}\label{hActionGl}
\begin{gather*}
h_{i}z^l \prod_{k=1}^n (v_{kk}^{(0)})^{\lambda_k}=\lambda_k(v_{ii}^{(0)})^{\lambda_i-1}v_{ii}^{(l)}\prod_{k \neq i }(v_{kk}^{(0)})^{\lambda_k},\\
\prod_{k=1}^n(v_{kk}^{(0)})^{\lambda_k}h_{i}z^l  =-\lambda_k(v_{ii}^{(0)})^{\lambda_i-1}v_{ii}^{(l)}\prod_{k \neq i }(v_{kk}^{(0)})^{\lambda_k}.
\end{gather*}
\end{lem}

Finally we need to check relations \eqref{TGLRelation}.
Recall the map \eqref{varphiDefinition}. Let $r'$ be the number such that $\lambda_{r'}=\min\{\lambda_r\}$.
\begin{lem}\label{hminimalAction}
\[\ker\varphi^{r'}_{\lambda_{r'}}\prod_{k=1}^n (v_{kk}^{(0)})^{\lambda_k}=\{0\}.\]
\end{lem}
\begin{proof}
    By definition we have:
\[h_{r'}z^l\prod_{k \neq r'} (v_{kk}^{(0)})^{\lambda_k}=0.\]
Therefore we need to prove
\[\ker\varphi^{r'}_{\lambda_{r'}}(v_{r'r'}^{(0)})^{\lambda_{r'}}=\{0\}.\]

We define the linear isomorphism
\[
\psi:\mathrm{span}\left\{ \prod_{i=0}^\infty(v_{r'r'}^{(i)})^{a_i} \right\}_{\sum_{i=0}^\infty a_i = \lambda_{r'}} \longrightarrow 
\Bbbk[s_1,\dots,s_{\lambda_{r'}}]^{\mathfrak{S}_{\lambda_{r'}}}
\]
(see \eqref{varphiDefinition})
defined by 
\[\psi\left(\prod_{i=1}^{\lambda_{r'}}v_{r'r'}^{(\nu_{i})}\right):=
\frac{1}{|Stab(\nu)|}\sum_{\sigma \in \mathfrak{S}_{\lambda_{r'}}}s_1^{\nu_{\sigma(1)}}\dots s_{\lambda_{r'}}^{\nu_{\sigma(\lambda_{r'})}}\]
 for any partition $\nu$.
Then it is easy to see that 
\[
\psi\circ h_{r'}z^l = (s_1^l+\dots + s_{r'}^l)\circ\psi,
\]
i.e. $\psi$ identifies the action of the Cartan elements  $h_{r'}z^l$ with the multiplication by the Newton sums $s_1^l+\dots + s_{r'}^l$.
\end{proof}

\begin{prop}\label{GlSurjectionSubquotient}
There exists the surjective map of $(\mathcal{I},\mathcal{I})$ bimodules
\[\mathbb{T}_\lambda \twoheadrightarrow\mathcal{F}_{\lambda}\left/ \sum_{\mu \succ \lambda} \mathcal{F}_{\mu}\right..\]
\end{prop}
\begin{proof}
By the construction the bimodule $\mathcal{F}_{\lambda}\left/ \sum_{\mu \succ \lambda} \mathcal{F}_{\mu}\right.$ is cyclic with the generator $\prod_{k=1}^n (v_{kk}^{(0)})^{\lambda_k}$. All the defining relations on this module are satisfied because of Lemmas \ref{extremalRelationsGl}, \ref{UDequationsGL} ,\ref{hActionGl}, \ref{hminimalAction}.
\end{proof}

\begin{thm}
 \[\mathbb{T}_\lambda \simeq\mathcal{F}_{\lambda}\left/ \sum_{\mu \succ \lambda} \mathcal{F}_{\mu}\right..\]   
\end{thm}
\begin{proof}
We have
\[E(x,y;q,0)=\sum_{N=0}^\infty \ch\, S^N(M_n[z]^+).\] 
Therefore thanks to equalities \eqref{NonsymmetricCauchy} and \eqref{aAlt} the surjection in Proposition~\ref{GlSurjectionSubquotient} is an isomorphism.
\end{proof}

\section{Functions on the Iwahori subgroup in $SL_n[z]$}\label{slfunctions}
In this section we work with the Lie algebra $\msl_n$. In particular, all the algebras, modules, etc defined in the whole generality in Section \ref{Current} are considered in type $A$ only. Our goal is to describe in details the bimodule of the space of functions on the Iwahori subgroup ${\bf I}\subset SL_n[z]$. We start with a discussion on the space of functions in the general setup.

\subsection{Generalities}
For an algebraic group $G$ we consider a faithful representation $V$ of dimension $n$. This gives an embedding of $G$ into the $n^2$-dimensional vector space $\mathrm{End}(V)$ 
with coordinates $u_{ij},\, i,j=1,\dotsc,n$. Then $\mathrm{End}(V)$ is naturally endowed with the structure of $G$-bimodule;
hence we obtain a structure of $G$-bimodule on the linear span of $u_{ij}$,\, $i,j=1,\dots,n$.
Let $f_1,\dots,f_m\in \Bbbk[u_{ij}]_{i,j=1,\dotsc,n}$ be a set of
generators of the ideal cutting out the image of the group $G$. Then the ideal generated by 
$f_1, \dots, f_m$ is preserved by the actions of the Lie algebra $\fg= {\rm Lie} (G)$.

Let $G[z]$ be the current group of the group $G$. 
We consider the set of variables $u_{ij}^{(k)}$, $i,j=1, \dots, n$, $k=0, 1, \dots$. We define $u_{ij}(s):=\sum_{k=0}^\infty u_{ij}^{(k)}s^k$ and the expansions
\[f_l(u_{ij}\mapsto u_{ij}(s))=\sum_{k =0}^\infty f_l^{(k)}s^k.\]
Then 
\[\Bbbk[G[z]]=\Bbbk[u_{ij}^{(k)}]/\langle f_l^{(k)}\rangle_{l=1,\dots,m,~ k=0, 1, \dots}.\]
 Note that the space $\Bbbk[G[z]]$ has a natural structure of $(\fg[z],\fg[z])$ bimodule, where 
 $\fg[z]$ is the Lie algebra of $G[z]$. Therefore the space $\Bbbk[G[z]]^*$ also has the natural structure of $(\fg[z],\fg[z])$-bimodule. 

Now we consider the Iwahori subgroup ${\bf I}\subset G[z]$. The space 
$\Bbbk[{\bf I}]^*\subset \Bbbk[G[z]]^*$ has the natural structure of $({\bf I},{\bf I})$-bimodule.
Given a monomial in variables $u_{ij}^{(k)}$ we refer to the sum of upper indices of the variables
in the monomial as the $z$-degree of the monomial. The corresponding $\bZ_{\ge 0}$ grading
on the polynomial ring $\Bbbk[u_{ij}^{(k)}]$ is called the $z$-grading. Note that all the relations $f_l^{(k)}$ are homogeneous with respect to this grading. Therefore the rings $\Bbbk[G[z]]$ and $\Bbbk[{\bf I}]$ inherit the $z$-grading.

\subsection{Character of the ring of functions}
 Let $u_{ij}^{(k)}\in (M_n[z]^+)^*$ be the coordinate functions on the space $M_n[z]^+$, i.e. $u_{ij}^{(k)}(v_{i'j'}^{(k')})=\delta_{i,i'}\delta_{j,j'}\delta_{k,k'}$.
We denote
\[u_{ij}(s):=\sum_{k=0}^\infty u_{ij}^{(k)}s^k,\]
where $u_{ij}^{(0)}=0$ if $i>j$. 

Let 
\[{\bf I}\subset SL_m[z]\subset M_n[z]^+\]
be the Iwahori sungroup. 
As a variety ${\bf I}$ is an affine subvariety
defined by the  relations
\[\prod_{i=1}^nu_{ii}(s) -1 =0.\]
Here the relation means that this variety is defined by the ideal generated by coefficients of the left hand side. Let $f^{(k)}$ be the coefficients in the expansion
\begin{equation}\label{deff}\prod_{i=1}^nu_{ii}(s) -1 =\sum_{k=0}^\infty f^{(k)} s^k.
\end{equation}

\begin{lem}\label{RegularSeq}
The elements $f^{(k)}$, $k=0, 1, \dots$ form a regular sequence.
\end{lem}
\begin{proof}
Clearly $f^{(0)}$ is not a zero divisor. Assume that $f^{(0)}, \dots, f^{(k)}$ is a regular sequence. Let $R$ be the polynomial ring in all the  variables $u_{ij}^{(l)}$ and let $R_k\subset R$ be the polynomial subring generated by $u_{ij}^{(l)}$ with $l\le k$.
Then $R=R_k[u_{ij}^{(l)}, l>k]$.
 The elements $f^{(0)}, \dots, f^{(k)}$ belong to $R_k$. Therefore
 \[R\left/ \langle f^{(0)}, \dots, f^{(k)} \rangle\right.= (R_k\left/\langle f^{(0)}, \dots, f^{(k)}\rangle\right.)[u_{ij}^{(l)}, l>k].\]
 Note that
 \[f^{(k+1)}=u_{11}^{(k+1)}u_{22}^{(0)}\dots u_{nn}^{(0)}+ u_{11}^{(0)}u_{22}^{(k+1)}\dots u_{nn}^{(0)}+\dots + u_{11}^{(0)}u_{22}^{(0)}\dots u_{nn}^{(k+1)}+a,\]
 where $a\in R_k$. Assume that $f^{(k+1)}$ is a zero divisor in the ring $R/\langle f^{(0)},\dots,f^{(k)} \rangle$, i.e.  $f^{(k+1)}b=0$ for some $b$. In particular, the highest degree term of $f^{(k+1)}b=0$ with respect to the variable $u_{11}^{(k+1)}$ vanishes; thus  $u_{22}^{(0)}\dots u_{nn}^{(0)}$ is a zero divisor in  $R_k/\langle f^{(0)},\dots,f^{(k)} \rangle$. However 
 \[u_{22}^{(0)}\dots u_{nn}^{(0)}=\left(u_{11}^{(0)}\right)^{-1} \mod f^{(0)},\]
i.e. the momomial $u_{22}^{(0)}\dots u_{nn}^{(0)}$ is invertible.
 Therefore $f^{(0)}, \dots, f^{(k)}, f^{(k+1)}$ is a regular sequence and we complete the proof by induction.
 \end{proof}

Note that the elements $f^{(k)}$ are homogeneous with respect to $z$-degree. Moreover they are left and right homogeneous with respect to the lattice $\bigoplus_{i=1}^{n-1}\mathbb{Z}(\varepsilon_i-\varepsilon_{i+1})\simeq\mathbb{Z}^n/\bZ(\varepsilon_1+\dots +\varepsilon_n)$.
The left and right weights  of $f^{(k)}$ are equal to $\varepsilon_1+\dots +\varepsilon_n$ and its $z$-degree is equal to $k$. Recall 
the $q$-Pochhammer symbol
\[
(a;q)_\infty = \prod_{k=0}^\infty (1-aq^k).
\]

\begin{prop}\label{SlFunctionsCharacter}
The character of the space $\Bbbk [{\bf I}]$ is equal to the image of 
  \[
  \frac{(x_1\cdots x_n y_1 \cdots y_n;q)_\infty}{\prod_{i\leq j}(x_iy_j;q)_\infty\prod_{i>j}(qx_iy_j;q)_\infty}
  \]
with respect to the $\overline{\phantom{a}}$ map, i.e., the left-hand side of \eqref{SlNonsymmetricCauchy}.
\end{prop}
\begin{proof}
    We have
\[\ch(\Bbbk [M_n[z]^+])=\frac{1}{\prod_{i\leq j}(x_iy_j;q)_\infty\prod_{i>j}(qx_iy_j;q)_\infty}.
\]
We consider the standard Koszul resolution (see e.g. \cite{Eis})
\begin{multline*}
\dots\rightarrow\bigoplus_{0 \leq k_1<\dots < k_l} f^{(k_1)}\dots f^{(k_l)}\Bbbk [M_n[z]^+]\rightarrow\\ \dots\rightarrow\bigoplus_{0 \leq k_1} f^{(k_1)}\Bbbk [M_n[z]^+]\rightarrow \Bbbk [M_n[z]^+]\rightarrow \Bbbk [{\bf I}],
\end{multline*}
because $f^{(0)},f^{(1)},\dots$ is a regular sequence. Now the alternating sum of $\msl_n$ characters of components of this resolution produces the desired expression for the character of $\Bbbk [{\bf I}]$. 
\end{proof}

\begin{cor}\label{SLIwahoriMacdonald}
    \[\ch(\Bbbk [{\bf I}])=  \sum_{\lambda\in P} 
\ch(\cA^{D}_{{\lambda}}) E_{\lambda}(X;q,0)E_{\lambda}(Y; q^{-1},\infty).
\]
\end{cor}
\begin{proof}
    It follows from equation \eqref{SlIwahoriCharacter}.
\end{proof}

\subsection{Filtration on $\Bbbk{\bf[I]}$}
In this subsection we construct an increasing filtration $\mathcal{G}_\la\subset \Bbbk[{\bf I}]$, $\la\in P$ of the space of functions on the Iwahori group and  describe the associated graded space.

Recall that the weight lattice $P=P(\msl_n)$ can be identified with the set $(\ZZ_{\ge 0})^n_0$ of nonnegative integer tuples with the smallest entry being zero  (since the restriction of the natural surjection $\ZZ^n=P(\mgl_n)\to P$ to $(\ZZ_{\ge 0})^n_0$ is one-to-one). 
For a weight $\la\in P$ we denote by $u_\la\in\Bbbk[{\bf I}]$ the product $\prod_{i=1}^n (u_{ii}^{(0)})^{\mu_i}$, where $\mu=(\mu_1,\dots,\mu_n)\in (\ZZ_{\ge 0})^n_0$ corresponds to $\la$. 

\begin{lem}\label{cogen}
The space $\Bbbk[{\bf I}]$ is cogenerated as a $(\mathcal{I},\mathcal{I})$-bimodule by vectors $u_\la$, $\la\in P$. 
\end{lem}
\begin{proof}
We first show that $\Bbbk[{\bf I}]$ is cogenerated by the subspace $\Bbbk[B]\subset \Bbbk[{\bf I}]$ ($\Bbbk[B]$ is contained in $\Bbbk[{\bf I}]$ as a subalgebra  generated by $u_{ij}^{(0)}$).   

Assume that there exists a function $\Psi\in\Bbbk[{\bf I}]$ such that 
\begin{equation}\label{UPsiU}
\U(\mathcal{I})\Psi \U(\mathcal{I}) \cap \Bbbk[B]= 0.
\end{equation}
We can (and will) assume that $\Psi$ is homogeneous of strictly positive $z$-degree. We will also
assume that this degree is the smallest possible (i.e. there is no function with the property \eqref{UPsiU}
of $z$-degree smaller than that of $\Psi$).
Then equation \eqref{UPsiU} implies that
$\Psi$ is invariant with respect to the left-right action of the product of groups
$\exp(z\fg[z])\times \exp(z\fg[z])$. In other words, for any $g_z\in {\bf I}$ and $(a,b)\in  \exp(z\fg[z])\times \exp(z\fg[z])$
one has $\Psi(g_z)=\Psi(ag_zb)$. 
In fact, since $\Psi $ is of positive $z$-degree and of the smallest $z$-degree with the property \eqref{UPsiU}, we conclude
$\U(z\fg[z])\Psi \U(z\fg[z])=0$, which is equivalent to the claim that $\Psi$ is invariant with respect to the product
of groups $\exp(z\fg[z])\times \exp(z\fg[z])$.

We note that the set
\[
\exp(z\fg[z])B\exp(z\mathfrak{g}[z])\subset {\bf I}
\]
is open dense.
For an element $g_z \in \exp(z\fg[z])B\exp(z\mathfrak{g}[z])$ we denote $g_z=a g_0 b$, $a, b \in \exp(z\mathfrak{g}[z])$, $g_0 \in B$.
Since the function $\Psi$ is invariant with respect to $\exp(z\fg[z])\times \exp(z\fg[z])$, we obtain:
\[
\Psi(g_z)=\Psi(ag_0b)=\left((a^{-1}\times b^{-1})\Psi\right) (g_0)=\Psi(g_0)=0
\]
because of positivity of $z$-degree of $\Psi$. Therefore $\Psi=0$.

We are left to show that $\Bbbk[B]$ is cogenerated as a $(\mathfrak{b},\mathfrak{b})$-bimodule by vectors $u_\la$, $\la\in P$. This is easily seen since $\Bbbk[B]$ is the span of products of the variables $u_{ij}^{(0)}$. 
\end{proof}

We now define the filtration $\mathcal{G}_\la\subset \Bbbk[{\bf I}]$ as follows:
\begin{equation*}
\mathcal{G}_\la = \left\{\Psi\in \Bbbk[{\bf I}]:\ \U(\mathcal{I})\Psi\U(\mathcal{I})\cap \mathrm{span} \{u_\mu,\mu\in P\}\subset \mathrm{span} \{u_\mu,\mu\preceq\la\}\right\}. 
\end{equation*}
Thanks to Lemma \ref{cogen}, $\mathcal{G}_\la$ is an increasing filtration on $\Bbbk[{\bf I}]$. Observe that, since the $(\mathcal{I},\mathcal{I})$-action on $\Bbbk[M_n[z]^+]$ preserves the degree in variables $u_{ij}^{(k)}$, the intersection in the definition of $\mathcal{G}_\la$ is finite-dimensional. 

The following lemma is easy to show and we omit the proof.

\begin{lem}
The quotient $\mathcal{G}_\la/\sum_{\mu\prec\la} \mathcal{G}_\mu$ has finite-dimensional homogeneous components with respect to the $z$-grading. The space $(\mathcal{G}_\la/\sum_{\mu\prec\la} \mathcal{G}_\mu)^\vee$ has  a natural structure of cyclic $(\mathcal{I},\mathcal{I})$-bimodule with cyclic vector $u^\vee_\la$. 
\end{lem}

Our goal is to show that $(\mathcal{G}_\la/\sum_{\mu\prec\la} \mathcal{G}_\mu)^\vee\simeq \mathbb{T}_\la$ (see Lemma \ref{Trel}). We know that both  $(\mathcal{G}_\la/\sum_{\mu\prec\la} \mathcal{G}_\mu)^\vee$ and $\mathbb{T}_\la$ are cyclic $(\mathcal{I},\mathcal{I})$-bimodules and that 
\[
\sum_{\la\in P} \ch (\mathcal{G}_\la/\sum_{\mu\prec\la} \mathcal{G}_\mu)^\vee =
\sum_{\la\in P} \ch \mathbb{T}_\la
\]
(since both left and right hand sides are equal to the character of $\Bbbk[{\bf I}]$). Hence it suffices to show the existence of the surjection 
\begin{equation}\label{TtoG}
\mathbb{T}_\la\to (\mathcal{G}_\la/\sum_{\mu\prec\la} \mathcal{G}_\mu)^\vee. 
\end{equation}

\subsection{Dual space of functions}
In order to construct the surjection \eqref{TtoG} we define subspaces $\mathcal{F}_\la\subset \Bbbk[{\bf I}]^*$ with the following properties: 
\begin{itemize}
\item $\mathcal{F}_\la\subset \mathcal{F}_{\mu}$ if $\la\succeq \mu$,
\item $\mathcal{F}_\la/\sum_{\mu\succ\la} \mathcal{F}_{\mu}$ is a cyclic $(\mathcal{I},\mathcal{I})$-bimodule,
\item there exists a surjection of $(\mathcal{I},\mathcal{I})$-bimodules $\mathbb{T}_\la\to \mathcal{F}_\la/\sum_{\mu\succ\la} \mathcal{F}_{\mu}$,
\item there exists a surjection of $(\mathcal{I},\mathcal{I})$-bimodules $\mathcal{F}_\la/\sum_{\mu\succ\la} \mathcal{F}_{\mu}\to (\mathcal{G}_\la/\sum_{\mu\prec\la} \mathcal{G}_{\mu})^\vee$.
\end{itemize}

To define $\mathcal{F}_\la$ and to prove the statements above we construct certain vectors $v_\mu\in \Bbbk[B]^*\subset \Bbbk[{\bf I}]^*$ coming from the $\mgl_n$ side of the story.   

We have the surjective map
$\Bbbk[M_n[z]^+]\twoheadrightarrow \Bbbk [{\bf I}]$ and the
dual embedding
$\Bbbk [{\bf I}]^*\hookrightarrow\Bbbk[M_n[z]^+]^*$.
Recall the explicit realization
\[\Bbbk[M_n[z]^+]=\Bbbk[u_{ij}^{(k)}:\ k\geq 0 \text{ if } i\leq j,\ k\geq 1 \text{ if }i> j].\]
One has the map
\[\Bbbk[M_n[z]^+]\twoheadrightarrow \Bbbk[M_n[z]^+]\left/\left\langle f^{(0)}, f^{(1)},\dots
\right\rangle\right.\]
and its dualization
\begin{equation}\label{restricted} 
\phi:\left(\Bbbk[M_n[z]^+]\left/\left\langle f^{(0)},f^{(1)},\dots
\right\rangle\right.\right)^*\hookrightarrow\left( \Bbbk[M_n[z]^+]\right)^*,
\end{equation}
where $\sum_{k\ge 0} f^{(k)}s^k=\prod_{i=1}^nu_{ii}(s) -1$ (see \eqref{deff}).
We denote by $v_{ij}^{(k)}\in (M_n[z]^+)^{\vee\vee}$ the dual linear
forms to $u_{ij}^{(k)}\in(M_n[z]^+)^{\vee}$ (we note that $(M_n[z]^+)^{\vee\vee}$ is isomorphic to $M_n[z]^+$). Then  we have the pairings between modules 
$S^N((M_n[z]^+)^{\vee})$ and $S^N(M_n[z]^+)$ defined on the bases by 
\[\left(\prod_{i,j,k}(u_{ij}^{(k)})^{t_{ij}^{k}},\prod_{i,j,k}(v_{ij}^{(k)})^{s_{ij}^{k}} \right)=\prod_{i,j,k}\delta_{t_{ij}^{k},s_{ij}^{k}}\binom{N}{\{t_{ij}^{k}\}}^{-1}.\]
Hence we have the same pairing between $\Bbbk[M_n[z]^+]$ and $S(M_n[z]^+)$.

For a composition $\nu=(\nu_1,\dots,\nu_n)\in\bZ_{\ge 0}^n$ recall the notation  $\overline\nu=\sum_{i=1}^{n-1} (\nu_i-\nu_{i+1})\omega_i$ for the corresponding $\msl_n$ weight. For an integral $\msl_n$ weight $\mu\in P$ we define
\[v_\mu:=\sum_{\nu:\ \overline\nu=\mu}\binom{\nu_1+ \dots+\nu_n}{\nu_1, \dots,\nu_n}\prod_{i=1}^n (v_{ii}^{(0)})^{\nu_i}.\]

\begin{lem}
For any $\mu\in P$
 \[v_\mu\in \phi\left(\Bbbk[M_n[z]^+]\left/\left\langle f^{(0)},f^{(1)},\dots\right\rangle\right.\right)^*.\] 
 One has $(v_\mu,u_\la)=\delta_{\mu,\la}$.
\end{lem}
\begin{proof}
    Observe that $v_\mu$ is automatically orthogonal to $\langle f^{(1)},f^{(2)},\dots\rangle$ for degree reasons. Thus we only need to prove that, for any monomial $u=\prod_{i,j}\prod_{k=1}^{r_{ij}}u_{ij}^{(t_{ij}^{k})}$, the element $u(1-\prod_{i=0}^n u_{ii}^{(0)})$ is orthogonal to $v_\mu$. Clearly the claim holds true if $r_{ij} \neq 0$ for some $i\neq j$ or $t_{ij}^{k} \neq 0$ for some $k$. Hence it suffices to consider 
    \[u=\prod_{i=1}^n (u_{ii}^{(0)})^{\eta_i}\]
    for some composition $\eta$.  One easily sees that for such monomials the scalar product of $u(1-\prod_{i=0}^n u_{ii}^{(0)})$ with $v_\mu$ vanishes unless $\bar\eta=\mu$.
    However
    \[u\left(1-\prod_{i=0}^n u_{ii}^{(0)}\right)=\prod_{i=1}^n (u_{ii}^{(0)})^{\eta_i}-\prod_{i=1}^n (u_{ii}^{(0)})^{\eta_i+1}\]
    and for $\eta$ such that $\bar\eta=\mu$ one has 
    \[\left(\prod_{i=1}^n (u_{ii}^{(0)})^{\eta_i},v_\mu \right)=1.\]
    Thus $v_\mu \in \phi\left(\Bbbk[u_{ij}^{(k)}]_{i \leq j}\left/\left\langle f^{(0)},f^{(1)},\dots\right\rangle\right.\right)^*$.
Finally, the equality $(v_\mu,u_\la)=\delta_{\mu,\la}$ follows from the explicit formulas.     
\end{proof}

We define the following subspaces $\mathcal{F}_\la\subset \Bbbk[{\bf I}]^*$:
\[
\mathcal{F}_\lambda=\sum_{\nu \succeq \lambda} \U(\mathcal{I})v_\nu\U(\mathcal{I}).\]
We have two obvious properties:
\begin{itemize}
\item  $\mathcal{F}_\lambda\subset \mathcal{F}_\mu$ whenever $\la\succ\mu$,
\item $\mathcal{F}_\lambda/\sum_{\mu\succ\la} \mathcal{F}_\mu$ is a cyclic $(\mathcal{I},\mathcal{I})$-bimodule with cyclic vector $v_\la$.
\end{itemize}

\begin{lem}\label{FtoG}
There is a natural pairing between $(\mathcal{I},\mathcal{I})$-bimodules:
\[
\left(\mathcal{G}_\lambda/\sum_{\mu\prec\la} \mathcal{G}_\mu\right)\times \left(\mathcal{F}_\lambda/\sum_{\mu\succ\la} \mathcal{F}_\mu\right) \to \Bbbk,
\]
which produces the surjective homomorphism
\begin{equation}\label{homo}
\mathcal{F}_\lambda/\sum_{\mu\succ\la} \mathcal{F}_\mu\to (\mathcal{G}_\lambda/\sum_{\mu\prec\la} \mathcal{G}_\mu)^\vee. 
\end{equation}
\end{lem}
\begin{proof}
From the definitions of $\mathcal{G}_\la$ and $\mathcal{F}_\la$ one has 
\[
(\mathcal{G}_\la, \sum_{\mu\succ\la} \mathcal{F}_\mu)=0,\
(\sum_{\mu\prec\la} \mathcal{G}_\mu,F_\la)=0.
\]
This produces the desired pairing and the homomorphism \eqref{homo} sending the cyclic vector $v_\la$ to the cyclic vector $u_\la^\vee$ (where  $u_\la^\vee$ is dual to $u_\la$ in the obvious sense). 
\end{proof}

We will now show that there exists a surjective homomorphism 
$\mathbb{T}_\la\to \mathcal{F}_\la/\sum_{\mu\succ\la} \mathcal{F}_{\mu}$.
Recall $\lambda=\sigma(\lambda_-)$, $\la\in P$, $\la_-\in P_-$.
\begin{prop}\label{slrel}
For all $\la\in P$ the following relations are satisfied:\\
\noindent 
for $\alpha\in\Delta_-$, $l\ge 0$
\begin{gather*}
e_{\widehat \sigma (\alpha+\delta)}z^l  v_\lambda\in \sum_{\mu \succ \lambda} \mathcal{F}_\mu,\\
v_\lambda e_{\widehat \sigma (-\alpha)}z^l\in \sum_{\mu \succ \lambda}\mathcal{F}_\mu;
\end{gather*}
for $\alpha \in \Delta_+$, $\sigma(\alpha)\in \Delta_-$:
\begin{gather*}
(e_{\widehat \sigma (\alpha)})^{\langle \alpha,\lambda \rangle} v_\lambda\in \sum_{\mu \succ \lambda}\mathcal{F}_\mu,\\
v_\lambda(e_{-\widehat \sigma (\alpha)})^{\langle \alpha,\lambda \rangle} \in \sum_{\mu \succ \lambda}\mathcal{F}_\mu;
\end{gather*}
for $\alpha \in \Delta_+$, $\sigma(\alpha)\in \Delta_+$:
\begin{gather*}
(e_{\widehat \sigma (\alpha)})^{\langle \alpha,\lambda \rangle+1} v_\lambda\in \sum_{\mu \succ \lambda}\mathcal{F}_\mu,\\
v_\lambda (e_{-\widehat \sigma (\alpha)})^{\langle \alpha,\lambda \rangle+1} \in \sum_{\mu \succ \lambda}\mathcal{F}_\mu.
\end{gather*}
Finally, for any $h\in\fh$
\[h z^k v_\lambda = -v_\lambda h z^k.\]
\end{prop}
\begin{proof}
Let us prove the first claim using  Lemma \ref{extremalRelationsGl} (the rest are proved in similar way using Lemma \ref{UDequationsGL} and Lemma \ref{hActionGl}). We are to show that
\[
e_{\widehat \sigma (\alpha+\delta)}z^l \sum_{\nu:\ \overline\nu=\la}\binom{\nu_1+ \dots+\nu_n}{\nu_1, \dots,\nu_n}\prod_{i=1}^n (v_{ii}^{(0)})^{\nu_i} \in \sum_{\mu \succ \lambda} \mathcal{F}_\mu.
\]

Let $\alpha\in\Delta_-$ and let $\sigma(\alpha)= \varepsilon_i -\varepsilon_j$.
Recall (see the proof of Lemma \ref{extremalRelationsGl}) that for any $l$
\[
E_{ij}z^l \prod_{k=1}^n (v_{kk}^{(0)})^{\lambda_k}= 
\frac{\lambda_j}{\lambda_i+1} \left((v_{jj}^{(0)})^{\lambda_j-1}(v_{ii}^{(0)})^{\lambda_i+1}\prod_{k \neq i,j} (v_{kk}^{(0)})^{\lambda_k}\right)E_{ij}z^l.
\]
Hence
\[
e_{\widehat \sigma (\alpha+\delta)}z^l v_{\la}=v_{\la+ \sigma (\alpha)}e_{\widehat \sigma (\alpha+\delta)}z^l.
\]
This implies the desired claim since $\lambda + \sigma(\al) \succ \lambda$ (because
$\al$ is negative and $\la=\sigma(\la_-)$).
\end{proof}

\begin{thm}
One has an isomorphism of of $(\mathcal{I},\mathcal{I})$-bimodules
\[\mathbb{T}_\lambda \simeq (\mathcal{G}_{\lambda}/ \sum_{\mu \prec \lambda} \mathcal{G}_{\mu})^\vee.\]
\end{thm}
\begin{proof}
Proposition \ref{slrel} and Lemma \ref{FtoG} imply that there exists the surjective map of $(\mathcal{I},\mathcal{I})$-bimodules
\begin{equation}\label{SLsurjection}
\mathbb{T}_\lambda \twoheadrightarrow (\mathcal{G}_{\lambda}/ \sum_{\mu \prec \lambda} \mathcal{G}_{\mu})^\vee.
\end{equation}
Now  Corollary \ref{SlIwahoriCharacter} tells us that the character of $\bigoplus_{\lambda \in P}\mathbb{T_\lambda}$ coincides with the character of 
    $\Bbbk[{\bf I}]^{\vee}$. Therefore the map \eqref{SLsurjection} is an isomorphism.
\end{proof}

\appendix\section{Example}\label{Appendix}
Let $\fg=\msl_2$. Then the Iwahori group ${\bf I}$ is the semidirect product  of the Borel
subgroup $B\subset SL_2$ and the unipotent group $\exp(\msl_2\T z\bC[z])$. Hence 
\begin{equation*}
{\rm ch}\, \bC[{\bf I}] =  {\rm ch}\, \bC[B]{{\rm ch}\, S^\bullet(\msl_2\T z\bC[z])}=
\frac{\sum_{i\in\bZ} X^iY^i}{(q)_\infty (XY^{-1};q)_\infty (qX^{-1}Y;q)_\infty}
\end{equation*}
where $X=X_1$ and $Y=Y_1$.
According to Proposition \ref{SlFunctionsCharacter} this expression is equal to
\[
\frac{(X^{\veb_1}X^{\veb_2}Y^{\veb_1}Y^{\veb_2};q)_\infty}{(X^{\veb_1}Y^{\veb_1};q)_\infty(X^{\veb_1}Y^{\veb_2};q)_\infty(X^{\veb_2}Y^{\veb_2};q)_\infty(qX^{\veb_2}Y^{\veb_1};q)_\infty}.
\]

The right hand side of the nonsymmetric $q$-Cauchy identity \eqref{SlNonsymmetricCauchy} contains the factors $a_\la(q)$, $E_\la(X;q,0)$ and $E_\la(Y;q^{-1},\infty)$ with $\la\in P$. In our case $\la\in\bZ$ and one has
\begin{gather*}
a_\la(q)=\begin{cases} 1/(q)_{-\la}, & \la\le 0,\\ 1/(q)_{\la-1}, & \la> 0\end{cases},\\
E_\la(X;q,0) = \begin{cases} r_{-\la}(X,q), & \la\le 0,\\ Xq^{\la-1}r_{\la-1}(Xq^{-1},q), & \la> 0\end{cases},\\
E_\la(Y;q^{-1},\infty) = \begin{cases} q^\la r_{-\la}(qY,q), & \la \le 0,\\ Yr_{\la-1}(Y,q), & \la > 0\end{cases},
\end{gather*}
where $r_\la(X,q)=\sum_{a=0}^\la X^{\la-2a}\frac{{(q;q)}_\la}{(q;q)_a(q;q)_{\la-a}}$ are the Rogers–Szeg\'o polynomials (see \cite{M2}).

Finally, let us write down the defining relations for modules $\mathbb{D}_\la$ and 
$\mathbb{U}^0_\la$, whose graded characters are $E_\lambda(X;q,0)$ and $E_\lambda(Y;q^{-1},\infty)$, respectively. Let $e,h,f$ be the standard basis of $\msl_2$.

For $\la\le 0$ the module $\mathbb{D}_\la$ contains cyclic vector $d_\la$ and the defining relations are 
\begin{gather*}
(h\otimes 1). d_{\la}=\la(h) d_{\la},\
e^{-\la+1}.d_\la=0,\ 
(f\T z^k).d_\la=0\ (k>0)
\end{gather*}
and $\mathbb{U}_\la$ is the quotient of $\mathbb{D}_\la$ by a single relation $e^{-\la}d_\la=0$. 
For $\la > 0$ the module $\mathbb{U}_\la$ contains cyclic vector $u_\la$ and the defining relations are 
\begin{gather*}
(h\otimes 1).u_{\la}=\la(h) u_{\la},\
(f\T z)^{\la+1}.u_\la=0,\ 
(e\T z^k). d_\la=0\ (k\ge 0)
\end{gather*}
and $\mathbb{D}_\la$ is the quotient of $\mathbb{D}_\la$ by a single relation $(f\T z)^{\la}.u_\la=0$.

\end{document}